\documentclass{aptpub}
\usepackage{amsmath,amsfonts,amssymb, comment}
\usepackage{mathrsfs}
\usepackage{mathscinet}
\usepackage{graphicx}
\usepackage{enumerate}
\usepackage{algorithm}
\usepackage[noend]{algpseudocode}
\usepackage{color}
\usepackage[numbers]{natbib}
\usepackage[hyperindex,breaklinks]{hyperref}
\usepackage{xspace}
\usepackage{multirow}
\usepackage[noend]{algpseudocode}
\def\BState{\State\hskip-\ALG@thistlm}
\hypersetup{
	colorlinks   = true,   urlcolor     = blue,   linkcolor    = blue,   citecolor    = blue }
\providecommand{\U}[1]{\protect\rule{.1in}{.1in}}
\hypersetup{
	pdftitle={},
	pdfsubject={Probability Theory},
	pdfauthor={Jose Blanchet, Fan Zhang},
	pdfdisplaydoctitle=true,
}
\authornames{J. BLANCHET AND F. ZHANG} 
\shorttitle{Exact Simulation} 
\begin{document}
	
\title{Exact Simulation for Multivariate It\^o Diffusions}
\authorone[Stanford University]{Jose Blanchet}
\authortwo[Stanford University]{Fan Zhang}
\addressone{Huang Engineering Center, 475 Via Ortega, Stanford, CA 94305, United States. }
\emailone{jose.blanchet@stanford.edu}
\addresstwo{Huang Engineering Center, 475 Via Ortega, Stanford, CA 94305, United States. }
\emailtwo{fzh@stanford.edu}
\begin{abstract}
We provide the first generic exact simulation algorithm for multivariate
diffusions. Current exact sampling algorithms for diffusions require the
existence of a transformation which can be used to reduce the sampling
problem to the case of a constant diffusion matrix and a drift which is the
gradient of some function. Such transformation, called Lamperti
transformation, can be applied in general only in one dimension. So,
completely different ideas are required for the exact sampling of generic
multivariate diffusions. The development of these ideas is the main
contribution of this paper. Our strategy combines techniques borrowed from
the theory of rough paths, on the one hand, and multilevel Monte Carlo on the
other.
\end{abstract}


\section{Introduction}

Consider a probability space $(\Omega,\mathcal{F},\mathbb{P})$ and an It\^{o}
Stochastic Differential Equation (SDE)
\begin{equation}  \label{eq:SDE}
dX(t)=\mu (X(t))dt+\sigma (X(t))dW(t),\text{ }X(0)=x_0,
\end{equation}
where $W(\cdot )$ is a $d^{\prime }$-dimensional Brownian motion under $%
\mathbb{P}$, and $\mu (\cdot) = (\mu_i(\cdot))_d:\mathbb{R}^{d}\rightarrow
\mathbb{R}^{d}$ and $\sigma (\cdot) = (\sigma_{ij}(\cdot))_{d\times
d^{\prime }}:\mathbb{R}^{d}\rightarrow \mathbb{R}^{d\times d^{\prime }}$
satisfy suitable regularity conditions. For instance, in order for %
\eqref{eq:SDE} to have a strong solution, it is sufficient to assume that
both $\mu \left( \cdot \right) $ and $\sigma \left( \cdot \right) $ are
uniformly Lipschitz.

Under additional regularity conditions on $\mu \left( \cdot \right) $ and $%
\sigma \left( \cdot \right) $, this paper provides the first Monte Carlo
simulation algorithm which allows sampling any discrete skeleton $X\left(
t_{1}\right) ,...,X\left( t_{m}\right) $ exactly, without any bias.

The precise regularity conditions that we impose on $\mu \left( \cdot
\right) $ and $\sigma \left( \cdot \right) $ are stated in Section \ref{sec:general}. In particular, it is sufficient for the validity of our Monte
Carlo method to assume $\mu \left( \cdot \right) $ and $\sigma \left( \cdot
\right) $ to be three times continuously differentiable, both with Lipschitz
continuous derivatives of order three. In addition, we must assume that $%
\sigma \left( \cdot \right) $ is uniformly elliptic.

Exact simulation of SDEs has generated a substantial amount of interest in
the applied probability and Monte Carlo simulation communities. The landmark
paper of \cite{beskos2005exact}, introduced what has become the standard
procedure for the design of generic exact simulation algorithms for
diffusions. The authors in \cite{beskos2005exact} propose a clever
acceptance-rejection sampler which uses Brownian motion as a proposal
distribution. The authors in \cite{chen2013localization} apply a
localization technique which eliminates certain boundedness assumptions
which are originally present in \cite{beskos2005exact}; see also \cite%
{beskos2006retrospective} for the use of retrospective simulation ideas to
dispense with boundedness assumptions.

The fundamental assumption underlying the work of \cite{beskos2005exact} and its
extensions is that the underlying (target) process has a constant diffusion
coefficient, i.e., $\sigma \left( x\right) =\sigma $ for every $x$. Beskos
and Roberts \cite{beskos2005exact} note that in the case $d=1$, owing to
Lamperti's transformation, the constant diffusion coefficient assumption
comes basically at no cost in generality.

Unfortunately, however, Lamperti's transformation is only generally
applicable in one dimension. In fact, \cite{ait2008closed} characterizes the
multidimensional diffusions for which Lamperti's transformation can be
successfully applied and these models are very restrictive.

Moreover, even if Lamperti's transformation is applicable in a
multidimensional setting, another implicit assumption in the application of
the Beskos and Roberts acceptance-rejection procedure is that the
drift coefficient $\mu \left( \cdot \right) $ is the gradient of some
function (i.e. $\mu \left( x\right) =\nabla v\left( x\right) $ for some $%
v\left( \cdot \right) $). This assumption, once again, comes at virtually no
cost in generality in the one-dimensional setting, but it may be very
restrictive in the multidimensional case.

Because of these limitations, a generic algorithm for exact simulation of
multidimensional diffusions, even under the regularity conditions that we
impose here, requires a completely different set of ideas.

The contribution in this paper is therefore not only the production of such a
generic exact simulation algorithm, but also the development of the ideas that are behind its
construction. In Section \ref{Sec_1_Identity} we introduce an algorithm which assumes a constant diffusion coefficient, but 
removes the assumption of the drift coefficient being the gradient of some function. In Section \ref{sec:general}, we eventually remove the requirement on a constant diffusion matrix and propose an algorithm applicable to general diffusions. The algorithms in Section \ref{Sec_1_Identity} and Section \ref{sec:general} are different in nature. However, they share some common elements, such as the use of so-called Tolerance-Enforced Simulation techniques based on rough path estimates. Even though the algorithm in Section \ref{sec:general} is more general, we believe that there is significant value in developing the algorithm in Section \ref{Sec_1_Identity} because of two reasons. The first one is pedagogical, the algorithm in Section \ref{Sec_1_Identity} is easier to understand while building on a key idea, which involves localizing essential quantities within specific compact domains with probability one. The second reason is that the algorithm in Section \ref{Sec_1_Identity}, being simpler, may be subject to potential improvement methodologies to be pursued in future research.

Potential improvements are particularly interesting directions specially given that, unfortunately, the algorithms that we present have infinite expected termination time. We recognize that this issue should be resolved for the algorithms to be widely used in practice, and we discuss the elements which lead to infinite expected running time in Section \ref{Section_Conclusion_Discussions}. There are basically two types of elements that affect the running time of the algorithm in Section \ref{sec:general}, one of them has to do with the types of issues that arise when trying to fully remove the bias in Tolerance-Enforced Simulation and related approximations, and the other issue has to do with the use of a density approximation coupled with Bernoulli factories. In contrast, the algorithm in Section \ref{Sec_1_Identity} is only affected by removing the bias in Tolerance-Enforced Simulation type approximations. We must stress, however, that the present paper shows for the first time that it is possible to perform exact sampling of multidimensional diffusions in substantial generality and, in doing so, it provides a conceptual framework different from the prevailing use of Lamperti transformation, which is the only available generic approach for producing exact sampling of diffusions.



Now, despite the algorithm's practical limitations, it is vital to
recognize the advantages that unbiased samplers have over biased samplers in the context of a massive parallel computing environment, because it is
straightforward to implement a parallel procedure to reduce the estimation
error of an unbiased sampler.

Recently, there have been several unbiased estimation procedures which have
been proposed for expectations of the form $\alpha =\mathbb{E}\left( f\left(
X\left( t\right) \right) \right) $, assuming $\text{Var}\left( f\left(
X\left( t\right) \right) \right) <\infty $. For example, the work of \cite%
{rhee2012new} shows that if $f\left( \cdot \right) $ is twice continuously
differentiable (with Lipschitz derivatives) and if there exists a
discretization scheme which can be implemented with a strong convergence
error of order 1, then it is possible to construct an unbiased estimator for
$\alpha $ with finite variance and finite expected termination time. The
work of \cite{giles2014antithetic} shows that such a discretization scheme
can be developed if $\mu \left( \cdot \right) $ and $\sigma \left( \cdot
\right) $ are sufficiently smooth under certain boundedness assumptions. The
paper \cite{henry2015unbiased} also develops an unbiased estimator for $%
\alpha $ using a regime-switching technique. Our work here is somewhat
related to this line of research, but an \textit{important difference} is
that \textit{we actually generate }$X\left( t\right) $ exactly, while
\textit{all} of the existing algorithms which apply in multidimensional
diffusion settings generate $Z$ such that $\mathbb{E}\left( Z\right) =\alpha
$. So, for example, if $f\left( \cdot \right) $ is positive, one cannot
guarantee that $Z$ is positive using the type of samplers suggested in \cite%
{rhee2012new}. However, by sampling $X\left( t\right) $ directly, one maintains the positivity of the estimator.

Another instance in which direct exact samplers are useful arises in the
context of stochastic optimization. For instance, consider the case in which
one is interested in optimizing a convex function of the form $g\left(
\theta \right) =\mathbb{E}\left( h\left( X\left( t\right) ,\theta \right)
\right) $, where $h\left( x,\cdot \right) $ is differentiable. In this case,
one can naturally construct an estimator $Z\left( \theta \right) $ such that
$g\left( \theta \right) =\mathbb{E}\left( Z\left( \theta \right) \right) $
using the results in \cite{rhee2012new} and optimize the mapping $\theta
\rightarrow $ $n^{-1}\sum_{i=1}^{n}Z_{i}\left( \theta \right) $, which
unfortunately will typically not be convex. So, having access to a direct
procedure to sample $X\left( t\right) $ in this setting is particularly
convenient as convexity is preserved.

The rest of the paper is organized as follows. In Section \ref%
{Sec_1_Identity}, we consider the case of multidimensional diffusions with a
constant diffusion coefficient and a Lipschitz continuous (suitably smooth)
drift. The general case is discussed in Section \ref{sec:general}, our
development uses localization ideas which are introduced in Section \ref%
{Sec_1_Identity}, but also some basic estimates of the transition density of
the underlying diffusion (e.g. Lipschitz continuity), these estimates are
developed in Appendix \ref{Appendix-Tech}. As mentioned before we discuss
the bottlenecks in the expected running time of the algorithm in Section \ref%
{Section_Conclusion_Discussions}.

\section{Exact Simulation of SDEs with Identity Diffusion Coefficient\label{Sec_1_Identity}}

In case that the Lamperti's transformation is applicable, the SDE of interest is
reducible to another SDE whose diffusion matrix is the identity. As a
result, it suffices to consider simulating the following SDE
\begin{equation}
dX(t)=\mu (X(t))dt+dW(t),\quad \quad X(0)=x_{0},  \label{eq:SDE-identity}
\end{equation}%
where $W=\left\{ W(t)=(W_{1}(t),\cdots ,W_{d}(t)):0\leq t<\infty \right\} $
is a $d$-dimensional Brownian motion. In this section we concentrate on the
identity diffusion case \eqref{eq:SDE-identity}, but the development can be
immediately extended to the case of a constant diffusion matrix. However,
throughout this section we must impose the following assumptions.

\begin{assumption}
\label{assumption:Lipchitz} The SDE \eqref{eq:SDE-identity} has a strong solution.
\end{assumption}

\begin{assumption}
\label{assumption:mu} The drift coefficient $\mu(\cdot)$ is three times
continuously differentiable.
\end{assumption}

Assumption \ref{assumption:mu} is the requirement of
TES, the theoretical foundation of our algorithm, which we shall introduce
later.

Let us introduce some notations first. For any set $G$ and $x\in \mathbb{R}%
^{d}$, we use $d(x,G)=\inf \left\{ \Vert x-y\Vert _{2}:y\in G\right\} $ to
denote the distance between $x$ and $G$; $\mathring{G}$ denotes the interior
of $G$; $\partial G$ denotes the boundary of $G$; $G^{c}$ denotes
complementary of $G$.

Consider a probability space $(\Omega ,\mathcal{F},\tilde{\mathbb{P}})$
endowed with a filtration $\left\{ \mathcal{F}_{t}:0\leq t\leq 1\right\} $,
and supporting a $d$-dimensional Brownian motion
\begin{equation*}
X(t)=(X_{1}(t),\cdots ,X_{d}(t));\quad 0\leq t\leq 1.
\end{equation*}

Let $\left\{ L(t):0\leq t\leq 1\right\} $ to be a $\tilde{\mathbb{P}}$-local
martingale defined as
\begin{equation}
L(t)=\exp \left( \int_{0}^{t}\mu ^{T}(X(t))dX(t)-\frac{1}{2}%
\int_{0}^{t}\Vert \mu (X(t))\Vert _{2}^{2}dt\right) ,
\label{Eqn-Likelihood-Ratio}
\end{equation}%
where $\mu ^{T}(\cdot )$ denotes the transpose of the column vector $\mu
(\cdot )$. Under Assumption \ref{assumption:Lipchitz}, $L(\cdot )$ is a $%
\tilde{\mathbb{P}}$-martingale, see Corollary 3.5.16 of \cite%
{karatzas2012brownian}.

In this case we can define a probability measure $\mathbb{P}$ through
\begin{equation*}
\mathbb{P}(A)=\mathbb{E}^{\tilde{\mathbb{P}}}\left[ I(A)L(1)\right] ;\quad
\quad \forall A\in \mathcal{F},
\end{equation*}%
where $I(A)$ denotes the indicator function of the set $A$ and $\mathbb{E}^{%
\tilde{\mathbb{P}}}\left( \cdot \right) $ is the expectation operator under $%
\tilde{\mathbb{P}}$.

Let
\begin{equation*}
W(t)=(W_{1}(t),\cdots ,W_{d}(t));\quad 0\leq t\leq 1
\end{equation*}%
be a $d$-dimensional process defined by
\begin{equation}
W(t)=X(t)-\int_{0}^{t}\mu (X(s))ds;\quad 0\leq t\leq 1.  \label{Eqn-Def-W}
\end{equation}%
The following theorem provides the distribution of $W(\cdot )$.

\begin{theorem}[Girsanov Theorem]
\label{thm:girsanov} If Assumption \ref{assumption:Lipchitz} is satisfied,
then the process $W(\cdot)$ is a $d$-dimensional Brownian motion on
probability space $(\Omega, \mathcal{F},\mathbb{P}).$
\end{theorem}

\begin{proof}
See, for instance, Theorem 3.5.1 of \cite{karatzas2012brownian}.
\end{proof}

It is readily apparent from \eqref{Eqn-Def-W} that $X(\cdot )$ is a weak
solution to SDE \eqref{eq:SDE-identity} on the probability space $(\Omega ,%
\mathcal{F},\mathbb{P})$. The exact simulation problem becomes sampling $%
X(1) $ under measure $\mathbb{P}$. Since $X(1)$ follows a normal
distribution under measure $\tilde{\mathbb{P}}$, we can attempt to use acceptance-rejection to
sample $X(1)$. A direct application of acceptance-rejection may proceed by using the $\tilde{%
\mathbb{P}}$ distribution of $X(1)$ (which is simply normal distribution) as the proposal, which, if acceptance-rejection is applicable,
would then be accepted with probability proportional to $L(1)$. However,
there are two obstacles when trying to apply such a direct acceptance-rejection approach.
First, the presence of the general stochastic integral appearing in the
definition of $L\left( 1\right) $ makes the likelihood ratio difficult to
directly compute. Second, a direct application of acceptance-rejection requires the likelihood
ratio, $L(1)$, to be bounded, which is unfortunately violated.

In order to deal with the first obstacle, we note that it is really not
necessary to accurately evaluate the likelihood ratio. In the standard procedure of acceptance-rejection, the likelihood ratio is only used for comparison with an independent
uniform random variable. Thus, to address the first obstacle, we can
approximate the likelihood ratio with a deterministic error bound, and keep
refining until we can decide to either accept or reject the proposal. It turns out, as we shall see in Corollary \ref{thm:localization}, that the same approximation technique can actually be used to localize $L\left( 1\right)$ and also resolve the second obstacle. Then, we will sample the distribution of $X(1)$ conditional on the localization of $L\left(1\right)$ using acceptance-rejection, where the rejection scheme is suggested by Lemma \ref{thm:girsanov}.

The theoretical foundation for such approximation and refinement strategy is
given by Tolerance-Enforced Simulation, which is presented in Theorem \ref%
{thm:TES}.

\begin{theorem}[Tolerance-Enforced Simulation]
\label{thm:TES} Consider a probability space $(\Omega ,\mathcal{F},\mathbb{P}%
)$ and the following SDE:
\begin{equation}
dY(t)=\alpha (Y(t))dt+\nu (Y(t))dW(t),\quad Y(t)=y_{0}  \label{Eqn-SDE-TES}
\end{equation}%
where $\alpha (\cdot )=(\alpha _{i}(\cdot ))_{d}:\mathbb{R}^{d}\rightarrow
\mathbb{R}^{d}$, $\nu (\cdot )=(\nu _{ij}(\cdot ))_{d\times d^{\prime }}:%
\mathbb{R}^{d}\rightarrow \mathbb{R}^{d\times d^{\prime }}$ and $W(\cdot )$
is a $d'$-dimensional Brownian motion. Suppose that $\alpha (\cdot )$ is
continuously differentiable and that $\nu (\cdot )$ is three times
continuously differentiable. Then, given any deterministic $\varepsilon >0$,
there is an explicit Monte Carlo procedure that allows us to simulate a
piecewise constant process $Y_{\varepsilon }(\cdot )$, such that
\begin{equation*}
\sup_{t\in \lbrack 0,1]}\Vert Y_{\varepsilon }(t)-Y(t)\Vert _{2}\leq
\varepsilon
\end{equation*}%
with probability one. Furthermore, for any $m>1$ and $0<\varepsilon
_{m}<\dots <\varepsilon _{1}<1$, we can simulate $Y_{\varepsilon _{m}}$
conditional on $Y_{\varepsilon _{1}},\dots ,Y_{\varepsilon _{m-1}}$.
\end{theorem}

\begin{proof}
See Theorem 2.1, Theorem 2.2 and Section 2.1 of \cite{blanchet2014epsilon},
where a detailed procedure of Tolerance-Enforced Simulation is also provided.
\end{proof}
\begin{remark}\label{remark-TES}
	
	Tolerance-Enforced Simulation is based on the L\'evy-Ciesielski Construction
	of the driving Brownian motion $W(\cdot)$ up to a random level. Consequently,
	$W(1)$ is obtained for free when we run TES in which a skeleton of the driving Brownian motion $W(\cdot)$ is simulated. In particular, for any $m>1$ and
	$0<\varepsilon_{m}<\dots <\varepsilon _{1}<1$, we can simulate $Y_{\varepsilon _{m}}$
	conditional on $Y_{\varepsilon _{1}},\dots ,Y_{\varepsilon _{m-1}}$ and $W(1)$.
	
\end{remark}
As a straightforward consequence of Theorem \ref{thm:TES}, we develop a
localization procedure of SDE in Corollary \ref{thm:localization}. Before
moving forward to state the result, we define some notations that will be
therein used.

\begin{definition}
A family of (Borel measurable) sets $\mathcal{G}=\{G_{i}\subset \mathbb{R}%
^{d}:i\in \mathbb{N}\}$ is said to be a countable continuous partition for a
$d$-dimensional random vector $Y$, if and only if

\begin{enumerate}
\item The sets in $\mathcal{G}$ are mutually disjoint, i.e. $G_{i}\cap G_{j}
= \emptyset$ for $i\neq j$;

\item $Y$ is concentrated on $\mathcal{G}$, that is $\mathbb{P}(Y\in \cup
_{i\in \mathbb{N}}G_{i})=1$;

\item $\mathbb{P}(Y\in \partial G_{i})=0,\forall i\in \mathbb{N}$.
\end{enumerate}

In addition, a function $\Xi_{\mathcal{G}}(x): \text{supp}(Y)\rightarrow
\mathbb{N}$ is defined such that $\Xi_{\mathcal{G}}(x) = i$ if and only if $%
x\in G_i$.
\end{definition}

\begin{corollary}
\label{thm:localization}Under the setting of Theorem \ref{thm:TES}, let $%
\mathcal{G}=\{G_{i}:i\in \mathbb{N}\}$ be a countable continuous partition
for $Y(1)$, then there is an algorithm for simulating $\Xi _{\mathcal{G}%
}(Y(1))$ that terminates in finite time with probability one. In particular,
for any set $G$ such that $\mathbb{P}(Y(1)\in \partial G)=0$, there is an
algorithm for simulating $I(Y(1)\in G)$.

\end{corollary}

\begin{proof}
Notice that $\mathbb{P}(Y(1)\in \partial G_i) = 0$, so $Y(1)\in
\bigcup_{i\in \mathbb{N}} \mathring{G}_i$ holds almost surely. Recalling
from Theorem \ref{thm:TES} that $\|Y_{\varepsilon}(1)-Y(1)\|_2\leq
\varepsilon$ a.s., which suggests that
\begin{equation*}
\mathbb{P}\left(\{\omega\in \Omega:Y(1)\in \mathring{G}_i\}\right) = \mathbb{%
P}\left(\bigcup_{\varepsilon>0} \{\omega\in \Omega:
d(Y_{\varepsilon}(1),G_i^c)>\epsilon \}\right).
\end{equation*}

Thus, we pick $\varepsilon \in (0,1)$ and apply TES to simulate the
approximation process $Y_{\varepsilon }(1)$. If
\begin{equation*}
d\Big(Y_{\varepsilon }(1),G_{\Xi _{\mathcal{G}}(Y_{\varepsilon
}(1))}^{c}\Big)>\varepsilon,
\end{equation*}
then $\Xi _{\mathcal{G}}(Y(1))=\Xi _{\mathcal{G}}(Y_{\varepsilon }(1))$,
which terminates the algorithm. Otherwise we keep refining the approximation
of TES, by setting $\varepsilon \leftarrow \varepsilon /2$, until $%
d\Big(Y_{\varepsilon }(1),G_{\Xi _{\mathcal{G}}(Y_{\varepsilon
}(1))}^{c}\Big)>\varepsilon$. The algorithm will ultimately terminate since
\begin{equation*}
\mathbb{P}\left( \bigcup_{i\in \mathbb{N}}\bigcup_{\varepsilon >0}\{\omega
\in \Omega :d(Y_{\varepsilon }(1),G_{i}^{c})>\epsilon \}\right) =1.
\end{equation*}%
The procedure for simulation of $I(Y(1)\in G)$ is just a particular case, by
setting $\mathcal{G}=\{G,G^{c}\}$. The details of the algorithm are given in
Algorithm \ref{algo:localization}.
\end{proof}

\begin{algorithm}
		\caption{Localization of SDE over Countable Continuous Partition $\mathcal{G}$}
		\label{algo:localization}
		\begin{algorithmic}[1]
			\State \textbf{Initialize} $\varepsilon\gets 1/2$.
			\State Apply TES to simulate $Y_{\varepsilon}(1)$,
			$i\leftarrow\Xi_\mathcal{G}(Y_\varepsilon(1))$.
			\While {$d(Y_{\varepsilon}(1),G_{i}^c)\leq\varepsilon$}
			\State Apply TES to simulate $Y_{\varepsilon/2}(1)$ conditional on $Y_{1/2}(1),\dots,Y_{\varepsilon}(1)$.
			\State $i\leftarrow\Xi_\mathcal{G}(Y_{\varepsilon/2}(1))$
			\State$\varepsilon\gets\varepsilon/2$.
			\EndWhile
			\State \textbf{Output} $i$.
		\end{algorithmic}
	\end{algorithm}

The algorithm for simulating $X(1)$ is performed in a two-stage fashion. At
first stage, the likelihood ratio $L(1)$ is localized with the help of
Corollary \ref{thm:localization}. (The efficiency of the algorithm may be slightly improved if we localize $X(1)$ and $L(1)$ simultaneously at the first step, then applying acceptance-rejection based on localization. However, this does not solve the problem of the infinite expected running time.) Then, at second stage, $X(1)$ is sampled
conditional on the result of localization.

We now illustrate how to localize $L(1)$ in detail. In order to write the
dynamics of $Y(1)$ in standard form as in \eqref{Eqn-SDE-TES}, we consider
the SDE of $(L(\cdot ),X(\cdot ))$ under measure $\mathbb{P}$ as follows,
\begin{align}  \label{Eqn-Likelihood-Ratio-Diff}
\begin{cases}
dL(t) = L(t)\|\mu(X(t))\|_2^2dt + L(t)\mu^{T}(X(t))dW(t), \\
dX(t) = \mu(X(t))dt + dW(t),%
\end{cases}%
\end{align}

Let $\mathcal{G} = \{G_i = [i,i+1)\times \mathbb{R}^d: i\in \mathbb{N} \}$ in
the rest of this section. As \eqref{Eqn-Likelihood-Ratio} guarantees that $%
L(1)$ is non-negative, it follows immediately that $\mathcal{G}$ is a
countable continuous partition for $L(1)$. Therefore, Algorithm \ref%
{algo:localization} is directly applicable to sample $\Xi_\mathcal{G}%
((L(1),X(1)))$ using SDE \eqref{Eqn-Likelihood-Ratio-Diff}. Without loss of
generality, we assume $\Xi_\mathcal{G}((L(1),X(1))) = i$ in the rest of this
section. It remains to sample $X(1)$ conditional on $\Xi_\mathcal{G}%
((L(1),X(1))) = i$ under probability measure $\mathbb{P}$.

The following lemma provides an alternative expression of the conditional
distribution of $X(1)$, which facilitates the simulation of $X(1)$ conditional on localization.

\begin{lemma}
\label{Lemma-Conditional-Probability} Let $U\sim\textrm{Unif}\;(0,1)$
independent of everything else under probability measure $\tilde{\mathbb{P}}$%
, then we have
\begin{equation*}
\mathbb{P}\Big(X(1)\in dx\Big|\Xi_\mathcal{G}\big((L(1),X(1))\big) = i\Big) = \tilde{\mathbb{P}}%
\Big(X(1)\in dx\Big|\max\big(i,(i+1)U\big)<L(1)<i+1\Big)
\end{equation*}
\end{lemma}

\begin{proof}
Due to the definition of conditional probability,
\begin{equation*}
\mathbb{P}\Big(X(1)\in dx\Big|\Xi _{\mathcal{G}}\big((L(1),X(1))\big)=i\Big)=\frac{\mathbb{P}%
\Big(X(1)\in dx;\Xi_{\mathcal{G}}\big((L(1),X(1))\big)=i\Big)}{\mathbb{P}\Big(\Xi_{\mathcal{G}}\big((L(1),X(1))\big)=i\Big)}.
\end{equation*}%
Recall that $d\tilde{\mathbb{P}}=L(1)d\mathbb{P}$, we have
\begin{equation*}
\mathbb{P}\Big(X(1)\in dx\Big|\Xi_{\mathcal{G}}\big((L(1),X(1))\big)=i\Big)=\frac{\mathbb{E}^{%
\tilde{\mathbb{P}}}\Big[L(1)I(X(1)\in dx;\Xi_{\mathcal{G}}\big((L(1),X(1)\big)=i\Big]}{%
\mathbb{P}\Big(\Xi_{\mathcal{G}}\big((L(1),X(1))\big)=i\Big)}.
\end{equation*}%
Since on $\Xi _{\mathcal{G}}((L(1),X(1))=i$,
\begin{equation*}
i\leq L(1)\leq i+1,
\end{equation*}%
we can rewrite the expectation into a probability by introducing $U\sim
\text{Unif}(0,1)$, namely,
\begin{align*}
& \mathbb{E}^{\tilde{\mathbb{P}}}\Big[L(1)I(X(1)\in dx;\Xi _{\mathcal{G}%
}\big((L(1),X(1))\big)=i)\Big] \\
=& (i+1)\tilde{\mathbb{P}}\Big(X(1)\in dx;\Xi _{\mathcal{G}%
}\big((L(1),X(1))\big)=i;(i+1)U<L(1)\Big) \\
=& (i+1)\tilde{\mathbb{P}}\Big(X(1)\in dx;\max \big(i,(i+1)U\big)<L(1)<i+1\Big).
\end{align*}%
By substitution, it follows easily that
\begin{align*}
\mathbb{P}\Big(X(1)\in dx\Big|\Xi _{\mathcal{G}}& ((L(1),X(1)))=i\Big)=\frac{(i+1)\tilde{%
\mathbb{P}}\Big(\max \big(i,(i+1)U\big)<L(1)<i+1\Big)}{\mathbb{P}\Big(\Xi _{\mathcal{G}%
}\big((L(1),X(1))\big)=i\Big)} \\
& \times \tilde{\mathbb{P}}\Big(X(1)\in dx\Big|\max \big(i,(i+1)U\big)<L(1)<i+1\Big).
\end{align*}%
It remains to prove that
\begin{equation*}
(i+1)\tilde{\mathbb{P}}\Big(\max \big(i,(i+1)U\big)<L(1)<i+1\Big)=\mathbb{P}\Big(\Xi _{\mathcal{G%
}}\big((L(1),X(1))\big)=i\Big).
\end{equation*}%
By a similar argument we can deduce that
\begin{align*}
& \mathbb{P}\Big(\Xi_{\mathcal{G}}\big((L(1),X(1)\big)=i\Big) \\
=& \mathbb{E}^{\tilde{\mathbb{P}}}\Big[L(1)I\big(\Xi_{\mathcal{G}}\big((L(1),X(1)\big)=i\big)\Big]
\\
=& (i+1)\tilde{\mathbb{P}}\Big(\Xi_{\mathcal{G}}\big((L(1),X(1)\big)=i;(i+1)U<L(1)\Big) \\
=& (i+1)\tilde{\mathbb{P}}\Big(\max \big(i,(i+1)U\big)<L(1)<i+1\Big),
\end{align*}%
which ends the proof.
\end{proof}

As a direct implication of Lemma \ref{Lemma-Conditional-Probability}, in
order to obtain an example sample for $X\left( 1\right) $ under $\mathbb{P}$%
, given $\Xi_{\mathcal{G}}\big((L(1),X(1)\big)=i$, we can simply simulate $X(1)$
conditional on $\max (i,(i+1)U)<L(1)<i+1$ under probability measure $\tilde{%
\mathbb{P}}$. In order to do this sampling under $\tilde{\mathbb{P}}$, we
can sample $U$ first and denote the value by $u$.
Then, observing that $X(\cdot)$ is the driving Brownian motion under the probability measure $\mathbb{P}$,
Algorithm \ref{algo:localization} is applied to the SDE
\begin{equation}  \label{Eqn-SDE-P-Prime}
dL(t)=L(t)\mu ^{T}(X(t))dX(t)
\end{equation}
to simulate the indicator function $I(\max (i,u)<L(1)<i+1)$. In addition, according to
Remark \ref{remark-TES}, when TES is employed in Algorithm \ref{algo:localization},
a sample of $X(1)$ is also produced simultaneously.
Thereafter, the value of $X(1)$ is accepted if and only if $I(\max (i,u)<L(1)<i+1)=1$;
otherwise we repeat the procedure in this paragraph, but we fix the parameter $i$,
because $\Xi _{\mathcal{G}}((L(1),X(1)))=i$ has already been sampled under
the correct distribution $\mathbb{P}$. The output of the algorithm, once the
value $X\left( 1\right)$ is finally accepted, follows the distribution of $X(1)$
under $\mathbb{P}$ without any bias.

We summarize the discussion in this section in the following theorem.
%
%

\begin{theorem}
\label{thm:constant-diffusion-exact-simulation} If Assumption \ref%
{assumption:Lipchitz} and \ref{assumption:mu} are satisfied, then there is
an exact simulation algorithm for $X(1)$ that terminates with probability
one, see Algorithm \ref{algo:constant-diffusion-exact-simulation}.
\end{theorem}

\begin{algorithm}
		\caption{Exact Simulation for SDE with Constant Diffusion Coefficient}
		\label{algo:constant-diffusion-exact-simulation}
		\begin{algorithmic}[1]
			\State Apply Algorithm \ref{algo:localization} to simulate random variable $\Xi_\mathcal{G}(L(1),W(1))$ associated with SDE (\ref{Eqn-Likelihood-Ratio-Diff}), record the result as $i \gets \Xi_\mathcal{G}(L(1),W(1))$.
			\Repeat
			\State Draw a sample $u$ from $\text{Unif}(0,i+1)$.
			\State Apply Algorithm \ref{algo:localization} to sample $I(\max(i,u)<L(1)<i+1)$ using SDE \eqref{Eqn-SDE-P-Prime}. The end of $\tilde{\mathbb{P}}$-Brownian path $x\gets X(1)$ is also sampled as a by-product of TES.
			\Until{$I(\max(i,u)<L(1)<i+1)=1$.}
			\State \textbf{Output} $x$ as a sample of $X(1)$.
		\end{algorithmic}
	\end{algorithm}

\section{Exact Simulation for General SDEs\label{sec:general}}

In this section, we will develop an exact simulation algorithm for the SDE %
\eqref{eq:SDE}. We shall fix $X\left( 0\right) =x_{0}$ and the dependence of
$x_{0}$ in some objects (such as the transition density of $X\left( 1\right)
$ will be omitted).

We are still going to construct an exact simulation algorithm based on acceptance-rejection in this section. However, for SDEs with non-constant diffusion matrix, applying Girsanov's theorem no longer provides a Brownian type proposal distribution for acceptance-rejection, so we will construct an acceptance-rejection algorithm based on the density of $X(1)$.

Throughout the rest of this section, we shall assume the following
assumptions and conditions.

\begin{assumption}
\label{assumption:TES} The drift coefficient $\mu (\cdot )$ is continuously
differentiable, and the diffusion coefficient $\sigma (\cdot )$ is three
times continuously differentiable. Moreover, a strong solution to SDE %
\eqref{eq:SDE} exists.
\end{assumption}

\begin{condition}
\label{condition:density} The probability distribution of $X(1)$ is
absolutely continuous with respect to Lebesgue measure. In other words, $%
X(1) $ has a density function denoted by $p(\cdot )$ with respect to the
Lebesgue measure.
\end{condition}

\begin{condition}
\label{condition:density-lipchitz} For any relatively compact set $S$, the
density $p(\cdot )$ is Lipschitz continuous with Lipschitz constant $C_{S}$,
i.e.
\begin{equation*}
|p(x)-p(y)|\leq C_{S}|x-y|\quad \quad \forall x,y\in S.
\end{equation*}
\end{condition}

\begin{condition}
\label{condition:density-lower-bound} For any relatively compact set $S$,
there exist $\delta_S > 0$ such that
\begin{equation*}
p(x)\geq \delta_S \quad\quad \forall x\in S.
\end{equation*}
\end{condition}

As we have seen in the previous section, Assumption \ref{assumption:TES} is
the necessary condition for the applicability of the TES result introduced
in Theorem \ref{thm:TES}, which enables us to strongly approximate $X(1)$.
Condition \ref{condition:density} will eventually be used to apply the acceptance-rejection
technique using an absolutely continuous (with respect to the Lebesgue
measure) proposal distribution. Conditions \ref{condition:density-lipchitz}
and \ref{condition:density-lower-bound}, as we shall see, will allow us to
control the bound of the likelihood ratio when applying acceptance-rejection.

It is important to ensure that the constants $C_{S}$ and $\delta _{S}$ are
explicitly computable in terms of $\mu \left( \cdot \right) $ and $\sigma
\left( \cdot \right) $ only, but we should also emphasize that we are not
assuming that the density $p(\cdot )$ is known.

There are many ways in which the computability of $C_{S}$ and $\delta _{S}$
can be enforced. For instance, in Appendix \ref{Appendix-Tech} we discuss a
set of assumptions involving classical estimators of the fundamental
solutions of parabolic equations, which we review in order to compute $C_{S}$
and $\delta _{S}$ explicitly.

The standard use of the acceptance-rejection algorithm requires knowing the density $p(x)$, which seems hopeless for the general SDE problem that we study. An alternative approach is constructing a non-negative, bounded, and unbiased estimator of $p(x)$. While the density $p(x)$ is unknown, an unbiased estimator of $p(x)$, denoted by $\Lambda_N(x)$ in Section \ref{Sec-Multilevel}, can be constructed by means of a local approximation of the density. However, the unbiased estimator $\Lambda_N(x)$ may be negative, so it cannot be directly used in acceptance-rejection. To remedy this problem, in Lemma \ref{Thm-Lambda-Plus} we construct a non-negative and unbiased estimator $\Lambda^{+}_N(x)$ of $p(x)$ using a random walk and a suitable Bernoulli factory. However, the estimator $\Lambda^{+}_N(x)$ is unbounded, so we propose to sample enough information about the SDEs (the ancillary variable $N'$), such that the estimator $\Lambda^{+}_N(x)$ conditional on the sampled information is locally bounded. Consequently, conditional on the localization of $X(1)$ and the additional information $N=N'$, the estimator $\Lambda^{+}_N(x)$ is bounded and non-negative.

We now state the outline of our exact simulation algorithm. First of all, we
apply a localization technique on the countable continuous partition $%
\mathcal{G}_{\text{loc}}$ defined as
\begin{equation*}
\mathcal{G}_{\text{loc}}=\{[i_{1},i_{1}+1)\times \cdots \times \lbrack
i_{d},i_{d}+1):(i_{1},\dots ,i_{d})\in \mathbb{Z}^{d}\}.
\end{equation*}%
Since $\mathcal{G}_{\text{loc}}$ has countable components, we can enumerate $%
\mathcal{G}_{\text{loc}}$ and rewrite it in terms of $\mathcal{G}_{\text{loc}%
}=\{G_{i}:i\in \mathbb{N}\}$, where $G_{i}$ is a unit hypercube. Obviously,
Algorithm \ref{algo:localization} is applicable to $X(1)$ with respect to
the countable continuous partition $\mathcal{G}_{\text{loc}}$.

Then, we introduce an ancillary random variable $N^{\prime }$ coupled with $%
X(1)$, and simulate $(N^{\prime }|X(1)\in G_{i})$. As we shall see, the
random variable $N^{\prime }$ will play an important role after we introduce
a suitable family of random variables whose expectations converge to the
density of $X\left( 1\right) $ at a given point. In the end, we will be able
to sample $X(1)$ conditional on $N^{\prime }$ and $X(1)\in G_{i}$, using the estimator $\Lambda^{+}_N(x)$, which is bounded and non-negative for $x\in G_{i}$ and $N=N^{\prime }$.%

The following theorem provides the main contribution of this paper.

\begin{theorem}
If Assumption \ref{assumption:TES} and Condition \ref{condition:density}-\ref%
{condition:density-lower-bound} are satisfied, then there is an algorithm
for exactly simulating $X(1)$ which terminates in finite time with
probability one; see Algorithm \ref{Algo-Exact-SDE}.
\end{theorem}

\begin{algorithm}
		\caption{Exact Simulation for Multivariate SDE}
		\label{Algo-Exact-SDE}
		\begin{algorithmic}[1]
			\State Simulate $\Xi_{\mathcal{G}_\text{loc}}(X(1))$ applying Algorithm \ref{algo:localization}. Set $i\gets\Xi_{\mathcal{G}_\text{loc}}(X(1))$.
			\State Simulate $(N'|X(1)\in G_i)$, denote the result by $n'$.
			\State Simulate $(X(1)|N' = n',X(1)\in G_i)$, denote the result by $x$.
			\State \textbf{Output} $x$.
		\end{algorithmic}
	\end{algorithm}

The rest of this section is organized as follows. Section \ref%
{Sec-Multilevel} applies a technique borrowed from Multilevel Monte Carlo to
construct the unbiased density estimator and the ancillary random variable $N^{\prime }$. Section \ref%
{Sec-Sample-N'} explains how to sample $N^{\prime }$ using acceptance-rejection and a suitable
Bernoulli factory \cite%
{nacu2005fast,latuszynski2011simulating,huber2016nearly} conditional on
localization. Section \ref{Sec-Cond-Sample-X} demonstrates how to sample $%
X(1)$ conditional on $N^{\prime }$, once again using a suitable localization.

\subsection{A Multilevel Representation of the Density\label{Sec-Multilevel}}

In this section, we borrow an idea from Multilevel Monte Carlo \cite%
{giles2008multilevel} to construct an unbiased estimator for $p(\cdot )$,
and we also introduce the ancillary random variable $N^{\prime }$.

In order to illustrate our idea, first we need to introduce some notations.
For any $x$ in $G_{i}$, we define $\{B_{r_{n}}(x):n\geq 1\}$ as a sequence
of open balls centered at $x$, whose radii $\{r_{n}:n\geq 1\}$, form a
decreasing sequence and $r_{n}\rightarrow 0$ as $n\rightarrow \infty $.

Let $V(r)$ denote the volume of a $d$-dimensional ball with radius $r$ (i.e.
the volume of $B_{r}\left( 0\right) $). We define $\overline{p_{n}}(x)$ to
be the average density over the ball $B_{r_{n}}(x)$, namely,
\begin{equation*}
\overline{p_{n}}(x)=[V(r_{n})]^{-1}\int_{B_{r_{n}}(x)}p(x)dx.
\end{equation*}%
Let $\hat{p}_{n}(x)$ denote a nonnegative unbiased estimator for $\overline{%
p_{n}}(x)$, i.e.
\begin{equation*}
\hat{p}_{n}(x)=[V(r_{n})]^{-1}\times I(X(1)\in B_{r_{n}}(x))
\end{equation*}%
for $n\geq 1$, where $\hat{p}_{n}(x)$ is defined using the same realization $X(1)$ for all $n$ and $x$. We define $\hat{p}_{0}(x):=0$ and $\overline{p_{0}}:=0$ for
notational simplicity. It follows immediately that $\mathbb{E}[\hat{p}%
_{n}(x)]=\overline{p_{n}}(x)$ for $n\geq 0.$

The density $p(x)$ is first decomposed into an infinite telescoping sums,
\begin{equation*}
p(x)=\sum_{n=0}^{\infty }\left( \overline{p_{n+1}}(x)-\overline{p_{n}}%
(x)\right) .
\end{equation*}%
Then, we introduce a randomization technique inspired by Randomized
Multilevel Monte Carlo (see \cite{mcleish2011general,rhee2015unbiased}). The
density $p(x)$ can be decomposed as expectation of an infinite sum of
estimators, which is truncated to a finite but random level so that the
expectation is invariant. The idea is to introduce an integer-valued random
variable $N$, which is independent of everything else. Then $p(x)$ can be
expressed as
\begin{align*}
p(x)& =\sum_{n=0}^{\infty }\left( \overline{p_{n+1}}(x)-\overline{p_{n}}%
(x)\right) \\
& =\sum_{n=0}^{\infty }\sum_{k=0}^{\infty }\frac{\left( \overline{p_{n+1}}%
(x)-\overline{p_{n}}(x)\right) }{\mathbb{P}(N\geq n)}\mathbb{P}(N=k)I(n\leq
k) \\
& =\sum_{k=0}^{\infty }\sum_{n=0}^{\infty }\frac{\left( \overline{p_{n+1}}%
(x)-\overline{p_{n}}(x)\right) }{\mathbb{P}(N\geq n)}\mathbb{P}(N=k)I(n\leq
k) \\
& =\sum_{k=0}^{\infty }\sum_{n=0}^{k}\frac{\left( \overline{p_{n+1}}(x)-%
\overline{p_{n}}(x)\right) }{\mathbb{P}(N\geq n)}\mathbb{P}(N=k) \\
& =\mathbb{E}\left[ \sum_{n=0}^{N}\frac{\left( \overline{p_{n+1}}(x)-%
\overline{p_{n}}(x)\right) }{\mathbb{P}(N\geq n)}\right] ,
\end{align*}%
where the third equality follows from Fubini's theorem, which can be
justified if
\begin{equation*}
\sum_{n=0}^{\infty }\sum_{k=0}^{\infty }\frac{\left\vert \overline{p_{n+1}}%
(x)-\overline{p_{n}}(x)\right\vert }{\mathbb{P}(N\geq n)}\mathbb{P}%
(N=k)I(n\leq k)=\sum_{n=0}^{\infty }\left\vert \overline{p_{n+1}}(x)-%
\overline{p_{n}}(x)\right\vert \leq 2C_{G_{i},r_1}\sum_{n=0}^{\infty
}r_{n}<\infty .
\end{equation*}%
We will show $\sum_{n=0}^{\infty }r_{n}<\infty $ in the sequel. Moreover, by
the tower property we have
\begin{align*}
\mathbb{E}\left[ \sum_{n=0}^{N}\frac{\left( \overline{p_{n+1}}(x)-\overline{%
p_{n}}(x)\right) }{\mathbb{P}(N\geq n)}\right] & =\mathbb{E}\left[
\sum_{n=0}^{N}\frac{\mathbb{E}[\hat{p}_{n+1}(x)-\hat{p}_{n}(x)|N]}{\mathbb{P}%
(N\geq n)}\right] \\
& =\mathbb{E}\left[ \mathbb{E}\left[ \left.\sum_{n=0}^{N} \frac{\hat{p}%
_{n+1}(x)-{\hat{p}_{n}}(x)}{\mathbb{P}(N\geq n)}\right\vert N\right] \right]
\\
& =\mathbb{E}\left[ \sum_{n=0}^{N}\frac{\hat{p}_{n+1}(x)-{\hat{p}_{n}}(x)}{%
\mathbb{P}(N\geq n)}\right] .
\end{align*}%
Therefore, if we define
\begin{equation*}
\Lambda _{n}(x)=\sum_{k=0}^{n}\frac{\hat{p}_{k+1}(x)-\hat{p}_{k}(x)}{\mathbb{%
P}(N\geq k)}\quad \mbox{for}\quad n\geq 0,
\end{equation*}%
it follows easily that%
\begin{equation}
p(x)=\mathbb{E}\left[ \Lambda _{N}(x)\right] .  \label{Eqn-Lambda-n}
\end{equation}%

We now are interested in obtaining bounds for $\Lambda _{n}(x)$ and its
expectation for $x\in G_{i}$. To this end we first define $G_{i,r_{1}}$ as $%
r_{1}$-neighborhood of set $G_{i}$, which consists of all points that at
a distance less than $r_{1}$ from $G_{i}$, i.e.
\begin{equation*}
G_{i,r_{1}}=\bigcup_{x\in G_{i}}B_{r_{1}}(x).
\end{equation*}%
It is not hard to observe that $G_{i,{r_{i}}}$ is a relatively compact set, to
which Conditions \ref{condition:density}-\ref{condition:density-lower-bound}
are applicable. In the following lemma, we will demonstrate that under such
conditions, one can judiciously pick the distribution of $N$ and the radii $%
\{r_{n}:n\geq 1\}$ in order to establish explicit bounds for $\Lambda
_{n}(x) $ and $\mathbb{E}[\Lambda _{n}(x)]$, respectively.

\begin{lemma}
\label{Thm-Bound} Suppose that $x\in G_i$ and Condition \ref%
{condition:density}-\ref{condition:density-lower-bound} are satisfied. If we
pick
\begin{equation*}
r_n = (3\delta_{G_{i,r_1}})/(2\pi^2n^3C_{G_{i,r_1}}) \quad\mbox{and}\quad
\mathbb{P}(N = n) = 1/[n(n+1)].
\end{equation*}
for $n\geq 1$, then we have
\begin{equation}
\delta_{G_{i,r_1}}/2\leq \mathbb{E}[\Lambda_{n}(x)] \leq [V(r_{1})]^{-1} +
\delta_{G_{i,r_1}}/2.
\label{eq-expectation-bound}
\end{equation}
and
\begin{equation}
|\Lambda_n(x)|\leq[V(r_1)]^{-1} +
\sum_{k=1}^n\left(k[V(r_{k+1})]^{-1}+k[V(r_k)]^{-1}\right)=:m_{n}.
\end{equation}
\end{lemma}

\begin{proof}
Let us construct the lower bound of $\mathbb{E}[\Lambda_{n}(x)]$ first. By
triangle inequality,
\begin{equation*}
\mathbb{E}[\Lambda_n(x)]\geq \mathbb{E}[\Lambda_{0}(x)] - \sum_{k=1}^n%
\mathbb{E}\left|\Lambda_{k}(x)-\Lambda_{k-1}(x)\right|.
\end{equation*}
On the one hand, from the definition of $\Lambda_0(x)$ and Condition \ref%
{condition:density-lower-bound}, we can conclude that
\begin{equation*}
\mathbb{E}[\Lambda_{0}(x)] = \mathbb{E}[\hat{p}_{1}(x)] = \overline{p_1} (x)
\geq \delta_{G_{i,r_1}}.
\end{equation*}
On the other hand, using the triangle inequality
\begin{align*}
	\mathbb{E}\left|\Lambda_{k}(x)-\Lambda_{k-1}(x)\right| &= \mathbb{E}
	\left\vert\frac{\hat{p}_{k+1}(x)-\hat{p}_{k}(x)}{\mathbb{%
			P}(N\geq k)}\right\vert\\
	&\leq\left\vert\frac{\mathbb{E}\hat{p}_{k+1}(x)-\mathbb{E}\hat{p}_{k}(x)}{\mathbb{%
			P}(N\geq k)}\right\vert\\
	&=(\mathbb{P}(N\geq k))^{-1} |\overline{p_{k+1}}(x) - \overline{p_{k}}(x)|.
\end{align*}
Then, from Condition \ref{condition:density-lipchitz} we have
\begin{equation}
|p(x) -p(y)|\leq C_{G_{i,r_1}}|x-y|\quad \mbox{for}\quad x,y\in C_{G_{i,r_1}}.
\label{eq-lip}
\end{equation}
Recall that $\overline{p_{k}}(x)$ is the average density over the ball $B_{r_{k}}(x)$ and $B_{r_{k}}(x)\subseteq B_{r_{1}}(x)\subseteq G_{i,r_1}$ for $x\in G_i$. It then follows form the \eqref{eq-lip} that
\begin{equation*}
	|\overline{p_{k+1}}(x) - \overline{p_{k}}(x)|\leq C_{G_{i,r_1}} \mathrm{diam}\big(B_{r_1}(x)\big) = 2C_{G_{i,r_1}}r_1.
\end{equation*}
Consequently,
\begin{equation*}
\sum_{k=1}^n\mathbb{E}\left|\Lambda_{k}(x)-\Lambda_{k-1}(x)\right| \leq
\sum_{k=1}^n 2C_{G_{i},r_{1}}(\mathbb{P}(N\geq k))^{-1}r_k \leq
\delta_{G_{i,r_1}}/2.
\end{equation*}
Combining the above inequality with $\mathbb{E}[\Lambda_{0}(x)]
\geq \delta_{G_{i,r_1}}$ yields
\begin{equation*}
\mathbb{E}[\Lambda_n(x)]\geq \delta_{G_{i,r_1}}/2.
\end{equation*}
Similarly, observing that $\mathbb{E}[\Lambda_{0}(x)] = \mathbb{E}[\hat{p}%
_1(x)] = [V(r_{1})]^{-1}\times \mathbb{P}(X(1)\in B_{r_{1}}(x))\leq [V(r_{1})]^{-1}$, for the upper bound we have
\begin{equation*}
\mathbb{E}[\Lambda_{n}(x)]\leq \mathbb{E}[\Lambda_{0}(x)] + \sum_{k=1}^n%
\mathbb{E}\left|\Lambda_{k}(x)-\Lambda_{k-1}(x)\right| \leq
[V(r_{1})]^{-1}+\delta_{G_{i,r_1}}/2.
\end{equation*}
We can also derive an upper bound of $|\Lambda_n(x)|$:
\begin{align*}
|\Lambda_n(x)|&\leq \sum_{k=0}^n(\mathbb{P}(N\geq k))^{-1}\left|\hat{p}%
_{k+1}(x)-\hat{p}_{k}(x)\right| \\
&\leq [V(r_1)]^{-1} +
\sum_{k=1}^n\left(k[V(r_{k+1})]^{-1}+k[V(r_k)]^{-1}\right)=:m_{n}<\infty,
\end{align*}
which ends the proof.
\end{proof}

In the rest of this paper, we will adopt the value of $r_{n}$ and the
distribution of $N$ in Lemma \ref{Thm-Bound}.

Even though we have constructed an unbiased estimator $\Lambda _{N}(x)$ for $%
p(x)$, acceptance-rejection is not applicable because $\Lambda _{N}(x)$ may be negative and
unbounded. In order to apply acceptance-rejection, we need a nonnegative unbiased estimator
for $p(x)$, which will be constructed in Lemma \ref{Thm-Lambda-Plus}. The
idea of such construction borrows an idea from \cite{fearnhead2010random}.
Let $\{\Lambda _{n,k}(x):k\geq 1\}$ be i.i.d. copies of $\Lambda _{n}(x)$.
We then define
\begin{equation*}
S_{n,k}(x):=\Lambda _{n,1}(x)+\dots +\Lambda _{n,k}(x).
\end{equation*}%
and
\begin{equation*}
\tau _{n}(x):=\inf \{k\geq 1:S_{n,k}(x)\geq 0\}.
\end{equation*}%
By Wald's first equation,
\begin{equation}  \label{Eqn-Wald}
\mathbb{E}[\Lambda _{n}(x)]=\mathbb{E}\left[ S_{n,\tau _{n}(x)}(x)\right] /%
\mathbb{E}[\tau _{n}(x)].
\end{equation}
Note that $S_{n,\tau _{n}(x)}(x)\geq 0$, but now we have an additional
contribution $1/\mathbb{E}[\tau _{n}(x)]$, which can be interpreted as a
probability. In order to sample a Bernoulli random variable with such probability, we will
need the following result which we refer to as the Bernoulli Factory.

\begin{theorem}[Bernoulli factory \protect\cite{nacu2005fast,huber2016nearly}%
]
\label{Thm-Bernoulli} Assume that $\epsilon \in (0,1/2]$ and $\alpha >0$ are
two known constants and that we have an oracle that outputs i.i.d.
Bernoullies with parameter $p\in (0,\left( 1-\epsilon \right) /\alpha ]$.
Then, there is an algorithm which takes the output of the oracle and
produces a Bernoulli random variable with parameter $f(p)=\min \left( \alpha
p,1-\epsilon \right) =\alpha p$. Moreover, if $\bar{N}\left( \alpha
,\epsilon \right) $ is the number of Bernoulli$\left( p\right) $ random
variables required to output Bernoulli$\left( f\left( p\right) \right) $
then $.004\cdot \alpha /\epsilon \leq \mathbb{E}\left( \bar{N}\left( \alpha
,\epsilon \right) \right) \leq 10\cdot \alpha /\epsilon $.
\end{theorem}

We now can explain how to construct $\Lambda _{n}^{+}(x)\geq 0$ such that $%
\mathbb{E}[\Lambda _{n}^{+}(x)]=\mathbb{E}[\Lambda _{n}(x)]$.

\begin{lemma}
\label{Thm-Lambda-Plus} There exists a family of random variables $\{\Lambda
_{n}^{+}(x):n\in \mathbb{N},x\in G_{i}\}$, such that the following
properties hold:

\begin{enumerate}
\item $0\leq \Lambda _{n}^{+}(x)\leq m_{n}$.

\item $\mathbb{E}[\Lambda _{n}^{+}(x)]=\mathbb{E}[\Lambda _{n}(x)]$.

\item Given $n$ and $x$, there is an algorithm for simulating $\Lambda
_{n}^{+}(x)$.
\end{enumerate}
\end{lemma}

\begin{proof}
Let $\bar{\Gamma}_{n}(x)$ be a Bernoulli random variable with parameter $%
\left( \mathbb{E}[\tau _{n}(x)]\right) ^{-1}$, and independent of everything
else. It follows that
\begin{equation*}
\mathbb{E}[\Lambda _{n}(x)]=\mathbb{E}\left[ \bar{\Gamma}_{n}(x)S_{n,\tau
_{n}(x)}(x)\right] .
\end{equation*}%
We write $\Lambda _{n}^{+}(x):=\bar{\Gamma}_{n}(x)S_{n,\tau _{n}(x)}(x)$.
Property 1 follows from the facts that $0\leq S_{n,\tau _{n}(x)}(x)\leq
\Lambda _{n,\tau _{n}(x)}(x)\leq m_{n}$, and that $0\leq \bar{\Gamma}%
_{n}(x)\leq 1$. Property 2 is justified directly by equation \eqref{Eqn-Wald}%
. To show that $\Lambda _{n}^{+}(x)$ can be simulated, we just need to
provide an algorithm for simulating $\bar{\Gamma}_{n}(x)$.

Recall that $\mathbb{E}[\Lambda _{n}(x)]\geq \delta _{G_{i,r_{1}}}/2$, we
have
\begin{equation*}
\mathbb{E}[\tau _{n}(x)]= \frac{\mathbb{E}[S_{n,\tau _{n}(x)}(x)]}{%
\mathbb{E}[\Lambda _{n}(x)]}\leq 2\delta _{G_{i,r_{1}}}^{-1}m_{n}.
\end{equation*}%
Consider Wald's second equation
\begin{equation*}
\mathbb{E}\left[ \left( S_{n,\tau _{n}(x)}(x)-\mathbb{E}(\Lambda
_{n}(x))\tau _{n}(x)\right) ^{2}\right] =\text{Var}(\Lambda _{n}(x))\mathbb{E%
}[\tau _{n}(x)],
\end{equation*}%
which implies
\begin{align}
\mathbb{E}\left[ (\tau _{n}(x))^{2}\right] & \leq \frac{2\mathbb{E}\left[
\tau _{n}(x)S_{n,\tau _{n}(x)}(x)\right]\mathbb{E}\left[ \Lambda _{n}(x)\right]+\text{Var}(\Lambda _{n}(x))\mathbb{E}\left[\tau
_{n}(x)\right] }{\mathbb{E}\left[ \Lambda _{n}(x)\right]^{2}
}\notag \\
& \leq \frac{(2m_{n}^{2}+m_{n}^{2})\mathbb{E}(\tau _{n}(x))}{(\delta
_{G_{i},r_{1}}/2)^{2}} \notag\\
& \leq \frac{3m_{n}^{3}}{(\delta _{G_{i},r_{1}}/2)^{3}}=:m_{\tau
,n}\label{eq-m-tau}.
\end{align}%
Now we shall proceed to simulate the random variable $\bar{\Gamma}_{n}(x)$.
Consider a random variable $T_{n}(x)$ with distribution
\begin{equation*}
\mathbb{P}(T_{n}(x)=k)=\mathbb{P}(\tau _{n}(x)\geq k)/\mathbb{E}[\tau
_{n}(x)]\quad \text{for}\quad k\geq 1.
\end{equation*}%
Since $I(T_{n}(x)=1)$ is the desired Bernoulli random variable with
parameter $(\mathbb{E}[\tau _{n}(x)])^{-1}$, it then suffices to simulate $%
T_{n}(x)$. Towards this end, we apply acceptance-rejection again using another random variable
$T^{\prime }$ as proposal, whose distribution is
\begin{equation*}
\mathbb{P}(T^{\prime }=k)=\frac{6}{\pi ^{2}k^{2}}\quad \text{for}\quad k\geq
1.
\end{equation*}

Since $\tau_n(x)$ is nonnegative, Markov's inequality asserts that
\begin{align}
	\mathbb{P}(\tau _{n}(x)\geq k) = \mathbb{P}\left(\left(\tau _{n}(x)\right)^2\geq k^2\right)\leq k^{-2}\mathbb{E}[\tau _{n}(x)^{2}]\leq k^{-2}m_{\tau,n},
	\label{eq-tau-prob-bound}
\end{align}
where the last inequality follows from \eqref{eq-m-tau}. Also, from the definition of $\tau_{n}(x)$ we know that $\mathbb{E}[\tau_n(x)]\geq 1$. Consequently,
the likelihood ratio between $T_{n}(x)$ and $T^{\prime }$ is given by
\begin{equation*}
\frac{\mathbb{P}(T_{n}(x)=k)}{\mathbb{P}(T^{\prime }=k)}=\frac{\pi ^{2}k^{2}%
\mathbb{P}(\tau _{n}(x)\geq k)}{6\mathbb{E}[\tau _{n}(x)]}\leq  \frac{\pi
^{2}}{6}m_{\tau ,n}.
\end{equation*}
From the above inequality we see
that the likelihood ratio is bounded, so the acceptance-rejection procedure is applicable.

Conditional on the proposal $T' = k$, we then introduce a new Bernoulli random variable $\tilde{\Gamma}_{n,k}(x)$ to
decide whether or not the proposal is accepted as $T_n(x)$. The distribution of $\tilde{\Gamma%
}_{n,k}(x)$ is defined as
\begin{equation*}
\mathbb{P}(\tilde{\Gamma}_{n,k}(x)=1)=1-\mathbb{P}(\tilde{\Gamma}_{n,k}(x)=0)=%
\frac{k^{2}}{2m_{\tau ,n}}
\mathbb{P}(\tau _{n}(x)\geq k).
\end{equation*}%
Hence it follows from \eqref{eq-tau-prob-bound} that
\begin{equation}  \label{Eqn-Gamma-Tilde-Upper-Bound}
\mathbb{P}(\tilde{\Gamma}_{n,k}(x)=1)\leq 1/2.
\end{equation}
Observe that the indicator $I(\tau _{n}(x)\geq k)$ is simulable,
but its distribution is not explicitly accessible, so it is natural to
sample $\tilde{\Gamma}_{n,k}(x)$ via Bernoulli factory introduced in Theorem %
\ref{Thm-Bernoulli}. Due to \eqref{Eqn-Gamma-Tilde-Upper-Bound}, the
function $f(\cdot )$ involved in Bernoulli factory is a linear function as
following
\begin{equation*}
f(p)=\min\left(\frac{k^{2}}{2m_{\tau ,n}}p,\frac{1}{2}\right) =\frac{k^{2}}{2m_{\tau ,n}}p.
\end{equation*}%
We summarize the procedure of simulating $\Lambda _{n}^{+}(x)$ in Algorithm %
\ref{Algo-Lambda-Plus}.
\end{proof}

\begin{algorithm}
		\caption{Simulation of $\Lambda_{n}^{+}(x)$.}
		\begin{algorithmic}[1]
			\Repeat
			\State Sample random variable $T'$, set $k\gets T'$.
			\Repeat
			\State Sample an independent copy of $I(\tau_{n}(x)\geq k)$
			as an input of Bernoulli factory associated with function
			\[
			f(p)=\min\left(\frac{k^{2}}{2m_{\tau ,n}}p,\frac{1}{2}\right)
			\]
			\Until{Bernoulli factory produces an output  $\gamma$.}
			\Until{$\gamma= 1$.}
			\If{k=1}
			\State Sample $S_{n,\tau_n(x)}(x)$ and \textbf{output} the result.
			\Else
			\State \textbf{Output} $0$.
			\EndIf
		\end{algorithmic}
		\label{Algo-Lambda-Plus}
	\end{algorithm}

We now introduce an ancillary random variable $N^{\prime }$ coupled with $%
X(1)$ in the following way,
\begin{equation}  \label{Eqn-Dist-Joint-N-X}
\mathbb{P}(N^{\prime }= n|X(1) \in dx) \propto \mathbb{P}(N=n)\times \mathbb{%
E}[\Lambda^{+}_n(x)].
\end{equation}

Assuming $(N^{\prime }|X(1)\in G_{i})$ can be simulated, $(X(1)|N^{\prime
},X(1)\in G_{i})$ can be easily simulated by acceptance-rejection as well due to the
convenient density representation given by \eqref{Eqn-X-Conditional}. The
algorithm for sampling $(N^{\prime }|X(1)\in G_{i})$ will be explained in
the next section. 

\subsection{Conditional Sampling of $N^{\prime }$\label{Sec-Sample-N'}}

In this section we will focus on the procedure for simulating $N^{\prime }$
conditional on $X(1)\in G_i$.

First we derive from equation \eqref{Eqn-Dist-Joint-N-X} the probability
mass function of $(N^{\prime }|X(1)\in G_i)$, namely
\begin{equation*}
\mathbb{P}(N^{\prime }= n|X(1)\in G_i) = \frac{\mathbb{P}(N=n)}{\mathbb{P}%
(X(1)\in G_i)} \times\int_{G_1} \mathbb{E}[\Lambda^+_n(x)]dx.
\end{equation*}
Due to the upper bound of $\mathbb{E}[\Lambda_n]$ given by Lemma \ref%
{Thm-Bound}, we have the following inequality
\begin{equation*}
\mathbb{P}(N^{\prime }= n|X(1)\in G_i) \leq \frac{\mathbb{P}(N=n)}{\mathbb{P}%
(X(1)\in G_i)}\times \left([V(r_1)]^{-1} + \delta_{G_{i,r_1}}/2\right).
\end{equation*}

Simulation of $(N^{\prime }|X(1)\in G_{i})$ can be achieved by acceptance-rejection. Consider
a Bernoulli random variable $\Gamma _{n}(x)$ defined as
\begin{equation}  \label{Eqn-Gamma-Upper-Bound}
\mathbb{P}(\Gamma _{n}(x)=1)=1-\mathbb{P}(\Gamma _{n}(x)=0)=\frac{1}{2}%
\left( [V(r_1)]^{-1}+\delta _{G_{i,r_{1}}}/2\right) ^{-1}\mathbb{E}[\Lambda
_{n}^{+}(x)]\leq \frac{1}{2}.
\end{equation}%
Then the outline of the acceptance-rejection algorithm for simulating $N^{\prime }$ would be:

\begin{description}
\item[Step 1] Sample $n$ from random variable $N$.

\item[Step 2] Sample $x$ from uniform distribution $U_{G_i}\sim\text{Unif}%
(G_i)$.

\item[Step 3] Simulate $\Gamma_{n}(x)$. If $\Gamma_{n}(x) = 0$, go to Step
1. Otherwise accept $n$ as a sample of $N^{\prime }$.
\end{description}

The only difficult step in the above procedure is Step 3, namely, sample $%
\Gamma _{n}(x)$, which we will discuss now.

Lemma \ref{Thm-Bound} implies that $0\leq \Lambda _{n}^{+}(x)\leq m_{n}$.
Let $U\sim \text{Unif}(0,m_{n})$, which is independent of everything else,
then $I(U\leq \Lambda _{n}^{+}(x))$ is a Bernoulli random variable with
parameter $(m_{n})^{-1}\mathbb{E}[\Lambda _{n}^{+}(x)]$. Due to %
\eqref{Eqn-Gamma-Upper-Bound}, Bernoulli factory given in Theorem \ref%
{Thm-Bernoulli} is applicable to simulate $\Gamma _{n}(x)$, using $%
I(U\leq\Lambda _{n}^{+}(x))$ as input and using the function
\begin{equation*}
f(p)=\min\left(\frac{m_{n}}{2}\left( [V(r_1)]^{-1}+\delta _{G_{i,r_{1}}}/2\right) ^{-1}p,%
\frac{1}{2}\right)= \frac{m_{n}}{2}\left( [V(r_1)]^{-1}+\delta _{G_{i,r_{1}}}/2\right)
^{-1}p.
\end{equation*}%
%
%
%

To conclude, we synthesize the complete steps for simulating $N^{\prime }$
in the following algorithm.

\begin{algorithm}
		\caption{Simulation of $(N'|x\in G_i)$}
		\begin{algorithmic}[1]
			\Repeat
			\State Sample $n$ from random variable $N$.
			\State Sample $x$ from uniform distribution $U_{G_i}\sim\text{Unif}(G_i)$.
			\Repeat
			\State Sample $u$ from $U\sim\text{Unif}(0,m_{n})$
			\State Sample $\lambda$ from distribution of $\Lambda_{n}^{+}(x)$ using Algorithm \ref{Algo-Lambda-Plus}.
			\State Use $I(u<\lambda)$ as an input of Bernoulli factory associated with function
			\[
			f(p) = \min\left(\frac{m_{n}}{2}\left( [V(r_1)]^{-1}+\delta _{G_{i,r_{1}}}/2\right) ^{-1}p,\frac{1}{2}\right)
			\]
			\Until{Bernoulli factory produces an output $\gamma$.}
			\Until{$\gamma = 1$}
			\State \textbf{Output} $n$.
		\end{algorithmic}
		\label{Algo-N-Prime}
	\end{algorithm}

\subsection{Sampling of $(X(1)|N^{\prime },X(1)\in G_i)$\label%
{Sec-Cond-Sample-X}}

In this section, we will focus on sampling $(X(1)|N^{\prime },X(1)\in G_i)$.

Without loss of generality, let us assume $N^{\prime }=n$ throughout the the
rest of this section. According to equation \eqref{Eqn-Dist-Joint-N-X} and
Lemma \ref{Thm-Lambda-Plus}, the conditional distribution of $X(1)$ can be
written as
\begin{equation}
\mathbb{P}(X(1)\in dx|N^{\prime }=n,X(1)\in G_{i})=C_{n,G_{i}}\mathbb{E}%
[\Lambda _{n}^{+}(x)],  \label{Eqn-X-Conditional}
\end{equation}%
where $C_{n,G_{i}}$ is a constant independent of $x$. Once again, as we
shall see $(X(1)|N^{\prime }=n,X(1)\in G_{i})$ can be be simulated by acceptance-rejection. We
use the uniform distribution $U_{G_{i}}$ as the proposal distribution, and
accept the proposal with probability $m_{n}^{-1}\times \mathbb{E}[\Lambda
_{n}^{+}(x)]$, so we can accept if and only if $I\left( U\leq \Lambda
_{n}^{+}(x)\right) =1$, where $U\sim \text{Unif}(0,m_n)$ is independent of
everything else. The output of the acceptance-rejection follows the desired distribution. The
explicit procedure for simulating $(X(1)|N^{\prime },X(1)\in G_{i})$ is
given in the following algorithm.
\begin{algorithm}
		\caption{Simulation of $X(1)$ Conditional on $N' = n, X(1)\in G_i$}
		\begin{algorithmic}[1]
			\Repeat
			\State Sample $x$ from uniform distribution $U_{G_i}\sim\text{Unif}(G_i)$.
			\State Sample $u$ from $U\sim\text{Unif}(0,m_{n})$
			\State Sample $\lambda$ from distribution of $\Lambda_{n}^{+}(x)$ using Algorithm \ref{Algo-Lambda-Plus}.
			\Until{$u\leq \lambda$}
			\State \textbf{Output} $x$.
		\end{algorithmic}
		\label{Algo-XT-Cond}
	\end{algorithm}

\section{Conclusion\label{Section_Conclusion_Discussions}}

The main contribution of this paper is the construction of the first generic
exact simulation algorithm for multidimensional diffusions. The algorithm
extensively uses several localization ideas and the application of TES
techniques. But it also combines ideas from multilevel Monte Carlo in order
to construct a sequence of random variables which ultimately provides an
unbiased estimator for the underlying transition density.

Although the overall construction can be implemented with a finite
termination time almost surely, the expected running time is infinite. Thus,
the contribution of the paper is of theoretical nature, showing that it is
possible to perform exact sampling of multivariate diffusions without
applying Lamperti's transformation. However, more research is needed in
order to investigate if the algorithm can be modified to be implemented in
finite expected time, perhaps under additional assumptions. Alternatively, perhaps some
controlled bias can be introduced while preserving features such as
positivity and convexity in the applications discussed at the end of the
Introduction. To this end, we conclude with a discussion of the elements in
the algorithm which are behind the infinite expected termination time.

There are three basic problems that cause the algorithm to have infinite
expected termination time. Two of them can be appreciated already from the
constant diffusion discussion and involves the use of the localization
techniques that we have introduced. The third problem has to do with the
application of the Bernoulli factory.

$\bullet\quad$\textbf{Problem 1:} The first problem arises when trying to
sample a Bernoulli of the form $I\left( X\left( 1\right) \in G\right) $.
Given $\epsilon _{n}>0$, sampling $X_{\epsilon _{n}}\left( 1\right) $ such
that $\left\Vert X_{\epsilon _{n}}\left( 1\right) -X\left( 1\right)
\right\Vert \leq \epsilon _{n}$ takes $O(\epsilon _{n}^{-\left( 2+\delta
\right) })$ computational cost for any $\delta >0$. So, if $G$ is a unit
hypercube in $d$ dimensions, using the density estimates for $X\left(
1\right) $, we obtain
\begin{equation*}
\mathbb{P}\left( d\left( X\left( 1\right) ,\partial G_{i}\right) \leq
\varepsilon \right) \geq c_{G}\varepsilon
\end{equation*}%
for some $c_{G}>0$. Therefore, if $N_{0}$ is the computational cost required
to sample $I\left( X\left( 1\right) \in G\right) $ we have that for some $%
\delta _{0}>0$
\begin{equation*}
\mathbb{P}\left(N_{0}>\frac{1}{\varepsilon ^{2+\delta}}\right) \geq \mathbb{P%
}\left( d\left(X\left( 1\right) ,\partial G_{i}\right) \leq \delta
_{0}\varepsilon \right) \geq c_{G}\delta _{0}\varepsilon .
\end{equation*}%
Therefore,
\begin{equation*}
\mathbb{P}\left( N_{0}>x\right) \geq c_{G}\delta _{0}\frac{1}{x^{1/(2+\delta
)}},
\end{equation*}%
which yields that $\mathbb{E}\left( N_{0}\right) =\infty $.

$\bullet\quad$\textbf{Problem 2:} The second problem arises in the acceptance-rejection step
applied in Lemma \ref{Lemma-Conditional-Probability}, which requires
sampling $X\left( 1\right) $ under $\tilde{\mathbb{P}}$ conditional on $\max
(i,(i+1)U)<L(1)<i+1$. Directly sampling from this conditional law might be
inefficient if $i$ is large. However, this problem can be mitigated using
rare-event simulation techniques, which might be available in the presence
of additional structure on the drift because under $\tilde{\mathbb{P}}\left(
\cdot \right) $, $X\left( \cdot \right) $ follows a Brownian motion.

$\bullet\quad$\textbf{Problem 3:} Arises because, as indicated in Theorem %
\ref{Thm-Bernoulli}, the computational complexity of the Bernoulli factory
of a linear function of the form $f\left( p\right) =\min \left( \alpha
p,1-\epsilon \right) $ scales in order $O\left( \alpha /\epsilon \right) $.
In our case, we are able to select $\epsilon =1/2$ and we invoke the
Bernoulli factory in Algorithms \ref{Algo-Lambda-Plus} and \ref{Algo-N-Prime}%
. In Algorithm \ref{Algo-Lambda-Plus}, $\alpha =O\left( k^{2}\right) $,
given $T^{\prime }=k$ and $\mathbb{E}\left( T^{\prime }\right) =\infty $. In
Algorithm \ref{Algo-N-Prime}, $\alpha =O\left( m_{n}\right) $, given $N=n$.
Although the bound which is used to define $m_{n}$ in Lemma \ref%
{Thm-Lambda-Plus} is far from optimal, in its current form, $m_{n}=O\left(
n^{3d+2}\right) $, we have that $\mathbb{E}\left( N\right) =\infty $.

\acks

We gratefully acknowledge support from the following NSF grants 1915967, 1820942, 1838676 as well as AFOSR.
\bibliographystyle{apt}

\clearpage
\appendix

\section{Transition Density Estimates \label{Appendix-Tech}}

In the appendix, we will discuss some assumptions which are sufficient for
the applicability of Conditions \ref{condition:density}, \ref%
{condition:density-lipchitz}, and \ref{condition:density-lower-bound}. In
addition, we also give explicit procedures for computing the constants which
appears in such Conditions.

Consider a matrix valued function $a(\cdot )=(a_{ij}(\cdot ))_{d\times d}:%
\mathbb{R}^{d}\mapsto \mathbb{R}^{d\times d}$ defined as
\begin{equation*}
a_{ij}(x):=\sum_{k=1}^{d^{\prime }}\sigma _{ik}(x)\sigma _{jk}(x)\quad \text{%
for}\quad 1\leq i,j\leq d.
\end{equation*}

\begin{assumption}
\label{Assumption-Bound} Every component of $\mu $ and $a$ are three times
continuously differentiable. Moreover, there exist a constant $M$ such that $%
\Vert \mu ^{(i)}\Vert _{\infty }\leq M$ and $\Vert a^{(i)}\Vert _{\infty
}\leq M$ for $i=0,1,2,3$.
\end{assumption}

\begin{assumption}
\label{Assumption-Uniform-Elliptic} There exist constants $%
0<\lambda_{\downarrow}<\lambda_{\uparrow}<\infty$, such that for all $x\in
\mathbb{R}^d$ and $\xi = (\xi_i)_{d}\in \mathbb{R}^d$, we have
\begin{equation*}
\lambda_{\downarrow}\|\xi\|_{2}\leq \sqrt{\xi^{T}a(x)\xi} \leq
\lambda_{\uparrow}\|\xi\|_{2}.
\end{equation*}
\end{assumption}

Under Assumption \ref{Assumption-Bound} and \ref{Assumption-Uniform-Elliptic}%
, it is proved in \cite{friedman2013partial} that the SDE \eqref{eq:SDE}
possesses a transition density denoted by $p(x,t;y,\tau )$, which satisfies
\begin{equation*}
\mathbb{P}(X(t)\in dx|X(\tau )=y)=p(x,t;y,\tau )dx
\end{equation*}%
for $\tau <t$. Therefore, Condition \ref{condition:density} is proved given
assumptions \ref{Assumption-Bound} and \ref{Assumption-Uniform-Elliptic}.

In the following Proposition, we will establish Condition \ref%
{condition:density-lipchitz} via Kolmogorov forward equation.

\begin{proposition}
Suppose Assumptions \ref{Assumption-Bound} and \ref%
{Assumption-Uniform-Elliptic} are satisfied, then for any relatively compact
set $S$, the density $p(\cdot )=p(\cdot ,T;x_{0},0)$ is Lipschitz continuous
with Lipschitz constant $C_{S}$, i.e.
\begin{equation*}
|p(x)-p(y)|\leq C_{S}\Vert x-y\Vert _{2}\quad \quad \forall x,y\in S.
\end{equation*}%
Furthermore, $C_{S}$ can be computed by Algorithm \ref{Algo-Lipchitz}.
\end{proposition}

\begin{proof}
Our methodology is closely related to parametrix method introduced in \cite%
{friedman2013partial}. Following the same scheme, we focus on explicit
computation of the constants.

Under Assumptions \ref{Assumption-Bound} and \ref%
{Assumption-Uniform-Elliptic}, $p(\cdot,\cdot;y,\tau)$ is a solution of
Kolmogorov forward equation, namely,
\begin{equation}  \label{Eqn-PDE}
\frac{\partial}{\partial t}p(x,t;y,\tau) = -\sum_{i =1}^d \frac{\partial}{%
\partial x_i}[\mu_i(x)p(x,t;y,\tau)] +\frac 1 2 \sum_{i=1}^d\sum_{j=1}^d%
\frac{\partial^2}{ \partial x_i\partial x_j}[a_{ij}(x)p(x,t;y,\tau)].
\end{equation}

We shall rewrite equation \eqref{Eqn-PDE} into its non-divergence form as
\begin{equation*}
\mathcal{L}_{f}p:=\sum_{i,j=1}^{d}a_{ij}(x)\frac{\partial ^{2}p(x,t;y,\tau )%
}{\partial x_{i}\partial x_{j}}+\sum_{i=1}^{d}b_{i}(x)\frac{\partial
p(x,t;y,\tau )}{\partial x_{i}}+c(x)p(x,t;y,\tau )-\frac{\partial
p(x,t;y,\tau )}{\partial t}=0,
\end{equation*}%
where
\begin{equation*}
b_{i}(x):=\sum_{j=1}^{d}\frac{\partial a_{ij}(x)}{\partial x_{j}}%
-\mu_{i}(x), c(x):=\frac{1}{2}\sum_{i=1}^{d}\sum_{j=1}^{d}\frac{\partial
^{2}a_{ij}(x)}{\partial x_{i}\partial x_{j}}-\sum_{i=1}^{d}\frac{\partial
\mu _{i}(x)}{\partial x_{i}},
\end{equation*}%
and $\mathcal{L}_{f}$ is a uniform parabolic operator. By Assumption \ref%
{Assumption-Bound}, it follows that
\begin{equation}  \label{Eqn-Bound-b-c}
\Vert b(x)\Vert _{\infty }\leq (d+1)M, \quad\quad|c(x)|\leq (0.5d^{2}+d)M.
\end{equation}%
We denote by $a^{-1}(x)=(a_{ij}^{-1}(x))_{d\times d}$ the inverse matrix of $%
(a_{ij}(x))_{d\times d}$, and define
\begin{equation*}
\theta (x,\xi ):=\sum_{i,j=1}^{d}a_{ij}^{-1}(\xi )(x_{i}-\xi _{i})(x_{j}-\xi
_{j}).
\end{equation*}%
Assumption \ref{Assumption-Uniform-Elliptic} implies that for all $\xi \in
\mathbb{R}^{d}$
\begin{equation*}
\lambda _{\uparrow }^{-1}\Vert \xi \Vert _{2}\leq \sqrt{\xi ^{T}a^{-1}(x)\xi
}\leq \lambda _{\downarrow }^{-1}\Vert \xi \Vert _{2}.
\end{equation*}%
Following the idea of the parametrix method, we also define a partial
differential equations with constant coefficients, namely,
\begin{equation*}
\mathcal{L}_{0}^{y}u:=\sum_{i,j=1}^{n}a_{ij}(y)\frac{\partial ^{2}u(x,t)}{%
\partial x_{i}\partial x_{j}}-\frac{\partial u(x,t)}{\partial t}=0.
\end{equation*}%
The fundamental solution of function $\mathcal{L}_{0}^{y}u=0$ is given by
\begin{equation*}
Z(x,t;\xi ,\tau )=C_{Z}(\xi )w(x,t;\xi ,\tau ),
\end{equation*}%
where
\begin{equation*}
C_{Z}(\xi ):=(2\sqrt{\pi })^{-d}[\det (a_{ij}(\xi ))]^{1/2}\leq (2\sqrt{\pi }%
)^{-d}\lambda _{\uparrow }^{d/2}=:C_{0},
\end{equation*}%
\begin{equation*}
w(x,t;\xi ,\tau ):=(t-\tau )^{-d/2}\exp \left( -\frac{\theta (x,\xi )}{%
4(t-\tau )}\right) .
\end{equation*}%
According to Theorem 1.3 and Theorem 1.10 in \cite{friedman2013partial}, $%
p(x,t;\xi ,\tau )$ can be represented by the parametrix method as%
\begin{equation*}
p(x,t;\xi ,\tau )=Z(x,t;\xi ,\tau )+\int_{\tau }^{t}\int_{\mathbb{R}%
^{d}}Z(x,t;y,s)\Phi (y,s;\xi ,\tau )dyds.
\end{equation*}%
where
\begin{align*}
\Phi (x,t;\xi ,\tau )& :=\sum_{k =1}^{\infty }(\mathcal{L}_{f}Z)_{k
}(x,t;\xi ,\tau ), \\
(\mathcal{L}_{f}Z)_{1}& :=\mathcal{L}_{f}Z, \\
(\mathcal{L}_{f}Z)_{k +1}& :=\int_{\tau }^{t}\int_{\mathbb{R}^{d}}\mathcal{L}%
_{f}Z(x,t;y,s)(\mathcal{L}_{f}Z)_{k }(y,s;\xi ,\tau )dyds.
\end{align*}%
for $k \geq 1$. Furthermore, the partial derivatives of the fundamental
solution admit the following expression,
\begin{equation}
\frac{\partial }{\partial x_{i}}p(x,t;\xi ,\tau )=\frac{\partial }{\partial
x_{i}}Z(x,t;\xi ,\tau )+\int_{\tau }^{t}\int_{\mathbb{R}^{d}}\frac{\partial
}{\partial x_{i}}Z(x,t;y,s)\Phi (y,s;\xi ,\tau )dyds.
\label{Eqn-Diff-Density}
\end{equation}%

Let us pick $\epsilon \in (0,1)$, then we can derive a bound for $Z$ as
\begin{equation*}
|Z(x,t;\xi ,\tau )|\leq C_{0}\times (t-\tau )^{-d/2}\exp \left( -\frac{%
(1-\epsilon )\Vert x-\xi \Vert _{2}^{2}}{4\lambda _{\downarrow }(t-\tau )}%
\right) .
\end{equation*}%
For the bound of $\partial Z(x,t;\xi ,\tau )/\partial x_{i}$, note that
\begin{equation*}
\left\vert \frac{\partial Z(x,t;\xi ,\tau )}{\partial x_{i}}\right\vert
=[4(t-\tau )]^{-1}\left\vert \frac{\partial \theta (x,\xi )}{\partial x_{i}}%
\right\vert C_{Z}(\xi )w(x,t;\xi ,\tau ),
\end{equation*}%
and that
\begin{equation*}
\left\vert \frac{\partial \theta (x,\xi )}{\partial x_{i}}\right\vert
=\left\vert 2\sum_{j=1}^{d}a_{ij}^{-1}(\xi )(x_{j}-\xi _{j})\right\vert \leq
2\lambda _{\uparrow }^{-1}\Vert x-\xi \Vert _{2}.
\end{equation*}%
Combining the definition of $C_Z(\cdot)$ and the above two equations imply
that
\begin{align*}
\left\vert \frac{\partial Z(x,t;\xi ,\tau )}{\partial x_{i}}\right\vert &
\leq \frac{\lambda _{\uparrow }^{-1}}{2}C_{0}|x-\xi |(t-\tau )^{-\frac{d+1}{2%
}}(\theta (x,\xi ))^{-1/2}\left[ \frac{\theta (x,\xi )}{t-\tau }\right]
^{1/2} \\
& \times \exp \left( -\frac{\epsilon \theta (x,\xi )}{4(t-\tau )}\right)
\exp \left( -\frac{(1-\epsilon )\theta (x,\xi )}{4(t-\tau )}\right) .
\end{align*}%
Applying the inequalities
\begin{equation*}
\left[ \frac{\theta (x,\xi )}{t-\tau }\right] ^{1/2}\exp \left( -\frac{%
\epsilon \theta (x,\xi )}{4(t-\tau )}\right) \leq \sup_{x\in \lbrack
0,+\infty )}x^{\frac{1}{2}}e^{-\frac{\epsilon x}{4}}=\left( \frac{2}{%
\epsilon e}\right) ^{1/2}
\end{equation*}%
and
\begin{equation*}
|x-\xi |(\theta (x,\xi ))^{-1/2}\leq \lambda _{\downarrow }^{1/2},
\end{equation*}%
we obtain
\begin{equation*}
\left\vert \frac{\partial Z(x,t;\xi ,\tau )}{\partial x_{i}}\right\vert \leq
\frac{C_{1}}{(t-\tau )^{\frac{d+1}{2}}}\exp \left( -\frac{(1-\epsilon )\Vert
x-\xi \Vert _{2}^{2}}{4\lambda _{\downarrow }(t-\tau )}\right)
\end{equation*}%
by setting
\begin{equation*}
C_{1}:=(2\epsilon e)^{-1/2}\lambda _{\downarrow }^{1/2}\lambda _{\uparrow
}^{-1}C_{0}.
\end{equation*}

Similarly, we can derive a bound for $\partial ^{2}Z(x,t;\xi ,\tau
)/\partial x_{i}\partial x_{j}$ and $\partial ^{2}Z(x,t;\xi ,\tau )/\partial
x_{i}^{2}$. For $i\neq j$,
\begin{align*}
\left\vert \frac{\partial ^{2}Z(x,t;\xi ,\tau )}{\partial x_{i}\partial x_{j}%
}\right\vert & \leq \frac{C_{2}}{(t-\tau )^{\frac{d+1}{2}}|x-\xi |}\exp
\left( -\frac{(1-\epsilon )\Vert x-\xi \Vert _{2}^{2}}{4\lambda _{\downarrow
}(t-\tau )}\right) , \\
\left\vert \frac{\partial ^{2}Z(x,t;\xi ,\tau )}{\partial x_{i}^{2}}%
\right\vert & \leq \frac{C_{3}}{(t-\tau )^{\frac{d+1}{2}}|x-\xi |}\exp
\left( -\frac{(1-\epsilon )\Vert x-\xi \Vert _{2}^{2}}{4\lambda _{\downarrow
}(t-\tau )}\right) .
\end{align*}%
where
\begin{align*}
C_{2}& :=C_{0}\left( \frac{4\lambda _{\downarrow }}{e\epsilon \lambda
_{\uparrow }}\right) ^{2}, \\
C_{3}& :=C_{0}\left( \frac{4\lambda _{\downarrow }}{e\epsilon \lambda
_{\uparrow }}\right) ^{2}+C_{0}\frac{M}{4}\left( \frac{2\lambda _{\downarrow
}}{\epsilon e}\right) ^{\frac{1}{2}}.
\end{align*}%
By definition of $Z(\cdot)$ we can observe that
\begin{align*}
\mathcal{L}_{f}Z(x,t;\xi ,\tau )& =\sum_{i,j=1}^{d}[a_{ij}(x)-a_{ij}(\xi )]%
\frac{\partial ^{2}Z(x,t;\xi ,\tau )}{\partial x_{i}\partial x_{j}} \\
& +\sum_{i=1}^{d}b_{i}(x)\frac{\partial Z(x,t;\xi ,\tau )}{\partial x_{i}}%
+c(x)Z(x,t;\xi ,\tau ).
\end{align*}%
Suppose $0\leq t-\tau \leq T$ in the sequel. By considering the upper bounds
of partial derivatives of $Z$, as well as \eqref{Eqn-Bound-b-c} and
Assumption \ref{Assumption-Bound}, we obtain
\begin{equation}
|\mathcal{L}_{f}Z(x,t;\xi ,\tau )|\leq \frac{C_{4}}{(t-\tau )^{\frac{d+1}{2}}%
}\exp \left( -\frac{(1-\epsilon )\Vert x-\xi \Vert _{2}^{2}}{4\lambda
_{\downarrow }(t-\tau )}\right) .  \label{Eqn-LZ}
\end{equation}%
where
\begin{equation*}
C_{4}:=dMC_{3}+d(d-1)MC_{2}+d(d+1)MC_{1}+T^{\frac{1}{2}}(0.5d^{2}+d)MC_{0}.
\end{equation*}%

Now, in order to find a bound for $\Phi (x,t;\xi ,\tau )$, we need to
introduce a technical lemma.

\begin{lemma}[Lemma 1.3 of \protect\cite{friedman2013partial}]
\label{Lemma-Integral} If $\beta $ and $\gamma $ are two constants in $%
(-\infty ,\frac{d}{2}+1)$, then
\begin{align*}
& \int_{\tau }^{t}\int_{\mathbb{R}^{d}}(t-s)^{-\beta }\exp \left( -\frac{%
h\Vert x-y\Vert _{2}^{2}}{4(t-s)}\right) (s-\tau )^{-\gamma }\exp \left( -%
\frac{h\Vert y-\xi \Vert _{2}^{2}}{4(s-\tau )}\right) dyds \\
& =\left( \frac{4\pi }{h}\right) ^{\frac{d}{2}}\text{Beta}\left( \frac{d}{2}%
-\beta +1,\frac{d}{2}-\gamma +1\right) (t-\tau )^{\frac{d}{2}+1-\beta
-\gamma }\exp \left( -\frac{h\Vert x-\xi \Vert _{2}^{2}}{4(t-\tau )}\right) ,
\end{align*}%
where Beta$(\cdot )$ is Beta function.
\end{lemma}

Due to \eqref{Eqn-LZ} and Lemma \ref{Lemma-Integral}, we can derive
\begin{align*}
| (\mathcal{L}_{f}Z)_2(x,t;\xi,\tau)|&\leq \int_\tau^t\int_{\mathbb{R}^d} |
\mathcal{L}_{f}Z(x,t;y,s)||\mathcal{L}_{f}Z(y,s;\xi,\tau)|dyds. \\
&\leq \frac{C_5C_6^2}{1!}(t-\tau)^{1-\frac{d}{2} }\exp\left(-\frac{%
(1-\epsilon)\|x-\xi\|_{2}^2}{4\lambda_{\downarrow}(t-\tau)}\right),
\end{align*}
where
\begin{align*}
C_5:= \left(\frac{4\pi\lambda_{\downarrow}}{1-\epsilon}\right)^{-\frac d 2},&
& C_6:= C_4\left(\frac{4\pi\lambda_{\downarrow}}{1-\epsilon}\right)^{\frac d
2}.
\end{align*}
By induction we can show that, for any positive integer $m$,
\begin{align*}
| (\mathcal{L}_{f}Z)_m(x,t;\xi,\tau)|\leq \frac{C_5C_6^m}{(m-1)!}
(t-\tau)^{m-\frac{d}{2}-1}\exp\left(-\frac{(1-\epsilon)\|x-\xi\|_{2}^2}{%
4\lambda_{\downarrow}(t-\tau)}\right).
\end{align*}
It turns out that
\begin{align*}
\Phi(x,t;\xi,\tau)&\leq\sum_{m=1}^\infty| (\mathcal{L }Z)_m(x,t;\xi,\tau)|
\notag \\
&\leq \frac{C_7}{(t-\tau)^{\frac{d}{2}}}\exp\left(-\frac{ (1-\epsilon)\|x-%
\xi\|_{2}^2}{4\lambda_{\downarrow}(t-\tau)}\right)
\end{align*}
where
\begin{align*}
C_7:= \sum_{m=1}^\infty\frac{C_5C_6^m}{(m-1)!} T^{m-1} =
C_{5}C_{6}\exp(C_{6}T)
\end{align*}
Recalling equation \eqref{Eqn-Diff-Density}, we can apply Lemma \ref%
{Lemma-Integral} again and conclude that
\begin{align*}
\left|\frac{\partial}{\partial x_i}p(x,t;\xi,\tau)\right| &\leq \left|\frac{
\partial}{\partial x_i}Z(x,t;\xi,\tau)\right|+\int_\tau^t\int_{\mathbb{R}^d}
\left|\frac{\partial}{\partial x_i}Z(x,t;y,s)\Phi(y,s;\xi,\tau)\right|dyds.
\notag \\
&\leq \left[\frac{C_1}{(t-\tau)^\frac{d+1}{2}}+\frac{C_8}{(t-\tau)^\frac{ d-1%
}{2}}\right]\exp\left(-\frac{(1-\epsilon)\|x-\xi\|_{2}^2}{
4\lambda_{\downarrow}(t-\tau)}\right),
\end{align*}
where
\begin{align*}
C_8:=2C_1C_7\left(\frac{4\pi\lambda_{\downarrow}}{1-\epsilon}\right)^{\frac
d 2}.
\end{align*}
Therefore, we obtain an upper bound for $|\nabla_x p(x,t;\xi,\tau)|$, by
considering
\begin{equation*}
|\nabla_x p(x,t;\xi,\tau)|\leq d\times\left|\frac{\partial}{\partial x_i}
p(x,t;\xi,\tau)\right|.
\end{equation*}
Therefore, for all $x,y \in S$ we have
\begin{equation*}
|p(x)-p(y)|\leq C_S\|x-y\|_2,
\end{equation*}
where
\begin{equation*}
C_S = \left[\frac{dC_1}{T^\frac{d+1}{2}}+\frac{dC_8}{T^\frac{d-1}{2}}\right]
\exp\left(-\frac{(1-\epsilon)\inf_{\bar{x}\in S}\|\bar{x}-x_0\|_2}{%
4\lambda_{\downarrow}T}\right)
\end{equation*}
\end{proof}

\begin{algorithm}
		\caption{Computation of Local Lipchitz Constant $C_S$}
		\begin{algorithmic}[1]
			\State \textbf{Input}: $M$ in Assumption \ref{Assumption-Bound}, $\lambda_{\downarrow}$ and $\lambda_{\uparrow}$ in Assumption \ref{Assumption-Uniform-Elliptic}, dimension $d$, time $T$, an arbitrary number $\epsilon\in(0,1)$.
			\State $C_0 \gets(2\sqrt{\pi})^{-d}\lambda_{\uparrow}^{d/2}$.
			\State $C_1 \gets (2\epsilon e)^{-1/2}\lambda_{\downarrow}^{1/2} \lambda_{\uparrow}^{-1}C_0.$
			\State $C_2\gets C_0\left(\frac{4\lambda_{\downarrow}}{e\epsilon\lambda_{\uparrow}}%
			\right)^{2}. $
			\State $C_3\gets C_0\left(\frac{4\lambda_{\downarrow}}{e\epsilon\lambda_{\uparrow}}%
			\right)^{2} +C_0 \frac{M}{4}\left(\frac{2\lambda_{\downarrow}}{\epsilon e }%
			\right)^\frac{1}{2}.$
			\State $C_4\gets dMC_3 + d(d-1)MC_2 + d(d+1)MC_1+T^\frac{1}{2}(0.5d^2+d)MC_0.$
			\State $C_5\gets \left(\frac{4\pi\lambda_{\downarrow}}{1-\epsilon}\right)^{-\frac d 2}.$
			\State $C_6\gets C_4\left(\frac{4\pi\lambda_{\downarrow}}{1-\epsilon}\right)^{\frac d 2}.$
			\State $C_7\gets C_{5}C_{6}\exp(C_{6}T).$
			\State $C_8\gets 2C_1C_7\left(\frac{4\pi\lambda_{\downarrow}}{1-\epsilon}\right)^{\frac d 2}.$
			\State $C_S\gets\left[\frac{dC_1}{T^\frac{d+1}{2}}+\frac{dC_8}{T^\frac{d-1}{2}}\right]
			\exp\left(-\frac{(1-\epsilon)\inf_{\bar{x}\in S}\|\bar{x}-x_0\|_2}{4\lambda_{\downarrow}T}\right).$
			\State \textbf{Output} $C_S$.
		\end{algorithmic}
		\label{Algo-Lipchitz}
	\end{algorithm}

Next, we will propose a computational procedure for lower bounds of
transition density. There is a substantial amount of literature that studies
lower bounds for the transition density of diffusions, through analytical
approaches or probabilistic approaches. For instance, Aronson \cite%
{aronson1967bounds} develops estimates of lower bounds of fundamental
solutions of second order parabolic PDEs in divergence form. Using Malliavin
calculus, Kusuoka and Stroock \cite{kusuoka1987applications} derived a lower
bound for the transition density of uniformly elliptic diffusions. Bally
\cite{bally2006lower} generalized the idea of \cite{kusuoka1987applications}
to locally elliptic diffusions. We follow the approach suggested by Sheu
\cite{sheu1991some} and review it in order to find explicit expressions to
obtain a computable lower bound.

In order to keep our paper self-contained, first we need to introduce some
notations used later.

Let $\mathcal{L}_{b}$ be the generator of Komolgorov backward equation:
\begin{equation*}
\mathcal{L}_{b}u(x,t):=\frac{1}{2}\sum_{i,j=1}^{n}a_{ij}(x)\frac{\partial
^{2}u(x,t)}{\partial x_{i}\partial x_{j}}+\sum_{i=1}^{d}\mu _{i}(x)\frac{%
\partial u(x,t)}{\partial x_{i}}-\frac{\partial u(x,t)}{\partial t}.
\end{equation*}%
The transition density as a function of $(x,t)\mapsto p(y,t;x,0)$ coincides
with the fundamental solution of Komolgorov backward equation:
\begin{align*}
\mathcal{L}_{b}u(x,t)& =0,\quad t>0,x\in \mathbb{R}^{d} \\
u(0,x)& =u_{0}(x).
\end{align*}

Throughout the rest of this section, we suppose that Assumptions \ref%
{Assumption-Bound} and \ref{Assumption-Uniform-Elliptic} are in force.

Recall that $a^{-1}(x)$ is the inverse matrix of $a(x)$, and define
\begin{align*}
k(x,\psi) = \frac{1}{2}\sum_{i,j=1}^d
a^{-1}_{ij}(x)(\mu_i(x)-\psi_i)(\mu_j(x) - \psi_j).
\end{align*}

For a fixed $y_{0}\in \mathbb{R}^{d}$, we define
\begin{equation*}
f^{\beta }(y;y_{0}):=\left( \frac{1}{\sqrt{2\pi \beta }}\right) ^{d}\frac{1}{%
\sqrt{\det a(y_{0})}}\exp \left( -\frac{1}{2\beta }%
\sum_{i,j=1}^{d}a_{ij}^{-1}(y_{0})(y-y_{0})_{i}(y-y_{0})_{j}\right) ,
\end{equation*}%
and
\begin{equation*}
p^{\beta }(y_{0},t;x,0):=\mathbb{E}_{x}[f^{\beta }(X(t);y_{0})].
\end{equation*}%
The continuity of the density implies
\begin{equation}
\lim_{\beta \rightarrow 0}p^{\beta }(y_{0},t;x,0)=p(y_{0},t;x,0).
\label{Eqn-p-beta}
\end{equation}%
For simplicity, we also define the logarithmic transform of $p$ and $%
p^{\beta }$ as
\begin{align*}
J(t,x)& :=-\log (p(y_{0},t;x,0)), \\
J^{\beta }(t,x)& :=-\log (p^{\beta }(y_{0},t;x,0)).
\end{align*}

To prepare the analysis, which is based on stochastic control, we introduce
the space of control functions by $\mathcal{F}_{T,x}$. The class $\mathcal{F}%
_{T,x}$ is defined as a family of measurable functions $\psi :[0,T]\times
\mathbb{R}^{d}\rightarrow \mathbb{R}^{d}$ such that the SDE
\begin{equation*}
d\eta (t)=\psi (t,\eta (t))dt+\sigma (\eta (t))dW(t),\quad \eta (0)=x
\end{equation*}%
has a weak solution $\eta (\cdot )$ that satisfies
\begin{equation*}
\mathbb{E}\left( \int_{0}^{T}\Vert \psi (t,\eta (t))\Vert _{2}^{2}dt\right)
<\infty .
\end{equation*}

Now we state a lemma that is crucial for proving the main result of this
section.

\begin{lemma}
\label{Lemma-J} Recall the definition of $\mathcal{F}_{T,x}$ and $%
\eta(\cdot) $ from previous paragraph, then we have
\begin{align*}
J^\beta(T,x) = \inf_{\psi\in \mathcal{F}_{T,x}} \mathbb{E}\left(\int_0^T
k(\eta(t),\psi(t))dt+J^{\beta}(0,\eta(T))\right).
\end{align*}
Together with \eqref{Eqn-p-beta}, we see that
\begin{align}  \label{Eqn-J}
J(T,x) = \lim_{\beta\rightarrow 0 } \inf_{\psi\in \mathcal{F}_{T-\beta,x}}
\mathbb{E}\left(\int_0^{T-\beta}
k(\eta(t),\psi(t))dt+J^{\beta}(0,\eta(T-\beta))\right).
\end{align}
\end{lemma}

\begin{proof}
See \cite{fleming2012deterministic}.
\end{proof}

\begin{theorem}
Suppose Assumptions \ref{Assumption-Bound} and \ref%
{Assumption-Uniform-Elliptic} are satisfied. Then, for any relatively
compact set $S$, the density $p(\cdot )=p(\cdot ,T;x_{0},0)$ has a uniform
lower bound $\delta _{S}>0$ in $S$, i.e.
\begin{equation*}
p(x)\geq \delta _{S}\quad \forall x\in S.
\end{equation*}%
Furthermore, $\delta _{S}$ can be computed by Algorithm \ref%
{Algo-Lower-Bound}.
\end{theorem}

\begin{proof}
Finding a lower bound of density $p(y_{0},T;x_{0},0)$ is equivalent to
finding an upper bound for $J(T,x_{0})$. Towards this end, it suffices to
find an upper bound for $J^{\beta }(T,x_{0})$ that is uniform in $\beta $.
We shall define $\phi (\cdot )$ as a linear function, such that $\phi
(0)=x_{0},\phi (T)=y_{0}$. Write
\begin{equation*}
\psi (t,x)=\frac{y_{0}-x_{0}}{T}-\frac{x-\phi (t)}{T-t},\quad 0\leq t\leq
T-\beta .
\end{equation*}%
It is not hard to see that $\psi \in \mathcal{F}_{T-\beta ,x_{0}}$.
Therefore,
\begin{equation}  \label{Eqn-J-beta}
J^{\beta }(T,x_{0})\leq \mathbb{E}\left( \int_{0}^{T-\beta }k(\eta (t),\psi
(t))dt+J^{\beta }(0,\eta (T-\beta ))\right) ,
\end{equation}%
according to lemma \ref{Lemma-J}. Notice that
\begin{equation}  \label{Eqn-Eta-Phi}
(\eta (t)-\phi (t))_{i}=(T-t)\sum_{l=1}^{d^{\prime }}\int_{0}^{t}\frac{1}{T-s%
}\sigma _{il}(\eta (s))dW_{l}(s),\quad \text{for}\quad i=1,\dots ,d.
\end{equation}
It follows that
\begin{equation*}
\mathbb{E}\left( (\eta (t)-\phi (t))_{i}(\eta (t)-\phi (t))_{j}\right)
=(T-t)^{2}\mathbb{E}\left( \int_{0}^{t}\frac{1}{(T-s)^{2}}a_{ij}(\eta
(s))ds\right)
\end{equation*}%
and
\begin{equation}  \label{Eqn-eta-phi_square}
\mathbb{E}(\Vert \eta (t)-\phi (t)\Vert _{2}^{2})=(T-t)^{2}\mathbb{E}\left(
\int_{0}^{t}\frac{1}{(T-s)^{2}}\sum_{i=1}^{d}a_{ii}(\eta (s))ds\right) \leq
d\lambda _{\uparrow }(T-t).
\end{equation}

We now apply a Taylor expansion of $k(\eta (t),\psi (t))$ around $(\phi (t),%
\dot{\phi}(t))$, where $\dot{\phi}(\cdot )$ denotes the derivative of $\phi
(\cdot )$. For notational simplicity, we define
\begin{align*}
\Delta_{1}(t)& =\eta (t)-\phi (t), \\
\Delta_{2}(t)& =\psi (t)-\dot{\phi}(t)=-\frac{1}{T-t}\Delta_{1}(t).
\end{align*}%
We also define
\begin{equation*}
D_{x_{i}}k(\lambda )=\frac{\partial }{\partial x_{i}}k(\phi (t)+\lambda
\Delta_{1}(t),\dot{\phi}(t)+\lambda \Delta_{2}(t))
\end{equation*}%
and similarly for $D_{\psi _{i}}k$, $D_{x_{i},x_{j}}k$, $D_{x_{i},\psi
_{j}}k $ and $D_{\psi _{i},\psi _{j}}k$. The Taylor expansion with
remainders of third order is given as following 
\begin{align*}
k(\eta (t),\psi (t)) = k_{0}(t) + k_{1}(t)+k_{2,1}(t)+k_{2,2}(t)+k_{2,3}(t)
+k_{3,1}(t)+k_{3,2}(t)+k_{3,3}(t),
\end{align*}
where
\begin{align*}
k_{0}(t) &:= k(\phi(t),\dot{\phi}(t)), \\
k_{1}(t) &:= \sum_{i=1}^{d}\left( D_{x_{i}}k(0)\Delta _{1,i}(t)+D_{\psi
_{i}}k(0)\Delta _{2,i}(t)\right), \\
k_{2,1}(t) &:= \frac{1}{2}\sum_{i,j=1}^{d} D_{x_{i},x_{j}}k(0)\Delta
_{1,i}(t)\Delta _{1,j}(t), \\
k_{2,2}(t) &:= \sum_{i,j=1}^{d} D_{x_{i},\psi _{j}}k(0)\Delta
_{1,i}(t)\Delta _{2,j}(t), \\
k_{2,3}(t) &:= \frac{1}{2}\sum_{i,j=1}^{d} D_{\psi _{i},\psi _{j}}k(0)\Delta
_{2,i}(t)\Delta _{2,j}(t), \\
k_{3,1}(t) &:= \sum_{i,j=1}^{d}\int_{0}^{1}\int_{0}^{1}\left(
D_{x_{i},x_{j}}k(\lambda \mu )-D_{x_{i},x_{j}}k(0)\right) \Delta
_{1,i}(t)\Delta _{1,j}(t)\lambda d\mu d\lambda, \\
k_{3,2}(t) &:= \sum_{i,j=1}^{d}2\int_{0}^{1}\int_{0}^{1}\left( D_{x_{i},\psi
_{j}}k(\lambda \mu )-D_{x_{i},\psi _{j}}k(0)\right) \Delta _{1,i}(t)\Delta
_{2,j}(t)\lambda d\mu d\lambda, \\
k_{3,3}(t) &:= \sum_{i,j=1}^{d}\int_{0}^{1}\int_{0}^{1}\left( D_{\psi
_{i},\psi_{j}}k(\lambda \mu )-D_{\psi _{i},\psi _{j}}k(0)\right) \Delta
_{2,i}(t)\Delta _{2,j}(t)\lambda d\mu d\lambda.
\end{align*}
Now we integrate all the above terms from $0$ to $T-\beta $ with respect to
variable $t$, then take expectations, and analyze the upper bounds of the
result term by term.

$\bullet $\textbf{\ Zeroth Order Term:} Notice that $k$ is in quadratic
form, with matrix $(a_{ij}^{-1}(x))$, so
\begin{equation*}
\mathbb{E}\left(\int_{0}^{T-\beta } k_{0}(t)dt\right)= \mathbb{E}\left(
\int_{0}^{T-\beta }k(\phi (t),\dot{\phi}(t))dt\right) \leq \lambda
_{\downarrow }^{-1}T\left( M+\frac{|y_{0}-x_{0}|}{T}\right) ^{2}.
\end{equation*}%

$\bullet $\textbf{\ First Order Terms:} We treat first order term $k_{1}(t)$
first. Noting that $\Delta _{2,i}(t)$ is a martingale due to %
\eqref{Eqn-Eta-Phi}, the first order terms
\begin{equation*}
\mathbb{E}\left(\int_{0}^{T-\beta } k_{1}(t)dt\right)= \mathbb{E}\left(
\int_{0}^{T-\beta }\left( D_{x_{i}}k(0)\Delta _{1,i}(t)+D_{\psi
_{i}}k(0)\Delta _{2,i}(t)\right) dt\right) =0.
\end{equation*}%

$\bullet $\textbf{\ Second Order Terms:} We then treat the second order
terms. As $D_{\psi _{i},\psi _{j}}k(0)=a_{ij}^{-1}(\phi (t))$,
\begin{align*}
\mathbb{E}\left(\int_{0}^{T-\beta } k_{2,3}(t)dt\right) =& \mathbb{E}\left(
\int_{0}^{T-\beta }\frac{1}{2}\sum_{i,j=1}^{d}D_{\psi _{i},\psi
_{j}}k(0)\Delta _{2,i}(t)\Delta _{2,j}(t)dt\right) \\
=& \frac{1}{2}\int_{0}^{T-\beta }\mathbb{E}\left( \int_{0}^{t}\frac{1}{%
(T-s)^{2}}\sum_{i,j=1}^{d}a_{ij}^{-1}(\phi (t))a_{ij}(\eta (s))ds\right) dt.
\end{align*}%
Writing
\begin{equation*}
a_{ij}(\eta (s))=(a_{ij}(\eta (s))-a_{ij}(\phi (s)))+(a_{ij}(\phi
(s))-a_{ij}(\phi (t)))+a_{ij}(\phi (t)),
\end{equation*}%
and noticing that $(a_{ij}(t))$ is symmetric, we see that
\begin{equation}  \label{Eqn-second_order1}
\frac{1}{2}\int_{0}^{T-\beta }\mathbb{E}\left( \int_{0}^{t}\frac{1}{(T-s)^{2}%
}\sum_{i,j=1}^{d}a_{ij}^{-1}(\phi (t))a_{ij}(\phi (t))ds\right) dt=\frac{d}{2%
}(\log (T)-\log (\beta )).
\end{equation}%
Assumption \ref{Assumption-Bound} implies the Lipschitz continuity of $%
a(\cdot )$, which gives,
\begin{align*}
& \frac{1}{2}\int_{0}^{T-\beta }\mathbb{E}\left( \int_{0}^{t}\frac{1}{%
(T-s)^{2}}\sum_{i,j=1}^{d}g_{ij}(\phi (t))(a_{ij}(\eta (s))-a_{ij}(\phi
(s)))ds\right) dt \\
\leq & \frac{1}{2}\int_{0}^{T-\beta }\mathbb{E}\left( \int_{0}^{t}\frac{M}{%
(T-s)^{2}}\sum_{i,j=1}^{d}|g_{ij}(\phi (t))|\times \Vert \eta (s)-\phi
(s)\Vert _{2}ds\right) dt \\
=& \frac{1}{2}\int_{0}^{T-\beta }\int_{0}^{t}\frac{M}{(T-s)^{2}}%
\sum_{i,j=1}^{d}|a_{ij}^{-1}(\phi (t))|\times \mathbb{E}\left( \Vert \eta
(s)-\phi (s)\Vert _{2}\right) dsdt.
\end{align*}%
Due to \eqref{Eqn-eta-phi_square} and Jensen's inequality,
\begin{equation*}
\mathbb{E}\left[ \Vert \eta (s)-\phi (s)\Vert _{2}\right] \leq \left(
d\lambda _{\uparrow }(T-t)\right) ^{1/2}.
\end{equation*}%
Observe that $\sum_{i,j=1}^{d}|a_{ij}^{-1}(\phi (t))|\leq d\lambda
_{\downarrow }^{-1}$, so we have
\begin{equation}  \label{Eqn-second_order2}
\begin{split}
& \frac{1}{2}\int_{0}^{T-\beta }\mathbb{E}\left( \int_{0}^{t}\frac{1}{%
(T-s)^{2}}\sum_{i,j=1}^{d}a_{ij}^{-1}(\phi (t))(a_{ij}(\eta (s))-a_{ij}(\phi
(s)))ds\right) dt \\
\leq & M(d\lambda _{\uparrow }T)^{1/2}d\lambda _{\downarrow }^{-1}.
\end{split}%
\end{equation}%
By the Lipschitz continuity of $(a_{ij}(\cdot ))$, it follows that
\begin{equation*}
|a_{ij}(\phi (s))-a_{ij}(\phi (t))|\leq MT^{-1}\Vert x_{0}-y_{0}\Vert
_{2}|s-t|.
\end{equation*}%
Therefore,
\begin{equation}  \label{Eqn-second_order3}
\begin{split}
& \frac{1}{2}\int_{0}^{T-\beta }\mathbb{E}\left( \int_{0}^{t}\frac{1}{%
(T-s)^{2}}\sum_{i,j=1}^{d}a_{ij}^{-1}(\phi (t))(a_{ij}(\phi (s))-a_{ij}(\phi
(t)))ds\right) dt \\
\leq & \frac{1}{2}d\lambda _{\downarrow }^{-1}MT^{-1}\Vert x_{0}-y_{0}\Vert
_{2}\int_{0}^{T-\beta }\left( \int_{0}^{t}\frac{t-s}{(T-s)^{2}}ds\right) dt
\\
\leq & \frac{1}{2}d\lambda _{\downarrow }^{-1}M\Vert x_{0}-y_{0}\Vert _{2}
\end{split}%
\end{equation}%
Combining (\ref{Eqn-second_order1}), (\ref{Eqn-second_order2}) and (\ref%
{Eqn-second_order3}) yields
\begin{equation*}
\begin{split}
\mathbb{E}\left(\int_{0}^{T-\beta } k_{2,3}(t)dt\right) &= \mathbb{E}\left(
\int_{0}^{T-\beta }\frac{1}{2}\sum_{i,j=1}^{d}D_{\psi _{i},\psi
_{j}}k(0)\Delta _{2,i}(t)\Delta _{2,j}(t)dt\right) \\
\leq & \frac{d}{2}(\log (T)-\log (\beta ))+M(d\lambda _{\uparrow
}T)^{1/2}d\lambda _{\downarrow }^{-1}+\frac{d}{2}\lambda _{\downarrow
}^{-1}M\Vert x_{0}-y_{0}\Vert _{2}.
\end{split}%
\end{equation*}%
By the chain rule and Assumption \ref{Assumption-Bound}, we obtain
\begin{align*}
|D_{x_{i}}a_{ij}^{-1}(x)|& \leq d^{2}\lambda _{\downarrow }^{-2}M, \\
|D_{x_{i}\psi _{j}}k(0)|& \leq \Psi _{1}(\Vert x_{0}-y_{0}\Vert _{2}/T), \\
|D_{x_{i}x_{j}}k(0)|& \leq \Psi _{2}(\Vert x_{0}-y_{0}\Vert _{2}/T).
\end{align*}%
where $\Psi _{i}(\cdot ):\mathbb{R}\rightarrow \mathbb{R};i=1,2$ are defined
as
\begin{align}
\Psi _{1}(x)&:=d^{2}\lambda _{\downarrow }^{-2}M(M+x)+d\lambda _{\downarrow
}^{-1}M,\label{Eqn-Psi-1}\\
\Psi _{2}(x)&:=(M+x)^{2}d^{2}(\frac{1}{2}\lambda _{\downarrow
}^{-2}Md+\lambda _{\downarrow }^{-3}M^{2}d^{2})+2\lambda _{\downarrow
}^{-1}M^{2}d^{2}+\lambda _{\downarrow }^{-1}Mdx\nonumber\\
&\qquad+2\lambda _{\downarrow
}^{-2}M^{3}d^{3}+2\lambda _{\downarrow }^{-2}M^{2}d^{2}x.\label{Eqn-Psi-2}
\end{align}%
Taking \eqref{Eqn-eta-phi_square} into consideration, we obtain,
\begin{align*}
\mathbb{E}\left(\int_{0}^{T-\beta } k_{2,2}(t)dt\right)&= \mathbb{E}\left(
\int_{0}^{T-\beta }\frac{1}{2}\sum_{i,j=1}^{d}D_{x_{i},\psi _{j}}k(0)\Delta
_{1,i}(t)\Delta _{2,j}(t)dt\right) \\& \leq \frac{1}{2}d\lambda _{\uparrow
}T\Psi _{1}(\Vert x_{0}-y_{0}\Vert _{2}/T), \\
\mathbb{E}\left(\int_{0}^{T-\beta } k_{2,1}(t)dt\right)&= \mathbb{E}\left(
\int_{0}^{T-\beta }\frac{1}{2}\sum_{i,j=1}^{d}D_{x_{i},x_{j}}k(0)\Delta
_{1,i}(t)\Delta _{1,j}(t)dt\right) \\& \leq \frac{1}{4}d\lambda _{\uparrow
}T^{2}\Psi _{2}(\Vert x_{0}-y_{0}\Vert _{2}/T),
\end{align*}

$\bullet $\textbf{\ Third Order Terms:} We proceed to analyze the third
order remainder terms. Let us consider $k_{3,3}(t)$ first. Notice that
\begin{equation*}
|D_{\psi _{i},\psi _{j}}k(\lambda \mu )-D_{\psi _{i},\psi _{j}}k(0)|\leq
\lambda _{\downarrow }^{-2}d^{2}M\lambda \mu \Vert \Delta_{1}(t)\Vert _{2},
\end{equation*}%
thus,
\begin{align*}
&\quad\left\vert\mathbb{E}\left(\int_{0}^{T-\beta } k_{3,3}(t)dt\right)\right\vert\\
&= \left\vert \mathbb{E}\left(
\sum_{i,j=1}^{d}\int_{0}^{1}\int_{0}^{1}\left( D_{\psi _{i},\psi
_{j}}k(\lambda \mu )-D_{\psi _{i},\psi _{j}}k(0)\right)
\Delta_{2,i}(t)\Delta _{2,j}(t)\lambda d\mu d\lambda \right) \right\vert \\
&\leq\frac{1}{6}Md^{3}\lambda _{\downarrow }^{-2}\mathbb{E}(\Vert \Delta
^{1}(t)\Vert _{2}\Vert \Delta_{2}(t)\Vert _{2}^{2}).
\end{align*}%
Then, by Burkholder-Davis-Gundy inequality,
\begin{align*}
& \mathbb{E}(\Vert \Delta_{1}(t)\Vert _{2}\Vert \Delta_{2}(t)\Vert
_{2}^{2})\leq (T-t)C_{\text{BDG}}(3)d^{\frac{1}{2}}\sum_{i=1}^{d}\mathbb{E}%
\left( \left( \int_{0}^{t}\frac{1}{(T-s)^{2}}a_{ii}(\eta (s))ds\right) ^{%
\frac{3}{2}}\right) \\
\leq & C_{\text{BDG}}(3)d^{\frac{3}{2}}\lambda _{\uparrow }^{\frac{3}{2}%
}(T-t)^{-\frac{1}{2}},
\end{align*}%
where $C_{BDG}(3)$ is the explicit constant in the Burkholder-Davis-Gundy
inequality. We can pick $C_{\text{BDG}}(p)=\left( \frac{p(p-1)}{2}(\frac{p}{%
p-1})^{p}\right) ^{p/2}$ (See Proposition 4.4.3 of \cite{revuz2013continuous}%
.). To summarize, we obtain
\begin{align*}
&\quad\left\vert\mathbb{E}\left(\int_{0}^{T-\beta } k_{3,3}(t)dt\right)\right\vert\\
&= \left\vert \int_{0}^{T-\beta }\mathbb{E}\left(
\sum_{i,j=1}^{d}\int_{0}^{1}\int_{0}^{1}\left( D_{\psi _{i},\psi
_{j}}k(\lambda \mu )-D_{\psi _{i},\psi _{j}}k(0)\right) \Delta
_{2,i}(t)\Delta _{2,j}(t)\lambda d\mu d\lambda \right) dt\right\vert \\
& \leq \frac{1}{3}C_{\text{BDG}}(3)Md^{\frac{9}{2}}\lambda _{\downarrow
}^{-2}\lambda _{\uparrow }^{\frac{3}{2}}T^{\frac{1}{2}}.
\end{align*}%
Next, we consider the other two remainders $k_{3,2}(t),k_{3,1}(t)$. Observe
that
\begin{equation*}
|D_{x_{i}\psi _{j}}k(\lambda \mu )|\leq \Psi _{1}(\Vert x_{0}-y_{0}\Vert
_{2}/T+\lambda \mu \Vert \Delta _{2}(t)\Vert _{2}),
\end{equation*}%
and
\begin{equation*}
|D_{x_{i}x_{j}}k(\lambda \mu )|\leq \Psi _{2}(\Vert x_{0}-y_{0}\Vert
_{2}/T+\lambda \mu \Vert \Delta _{2}(t)\Vert _{2}).
\end{equation*}%
Thus, by a similar argument, we can also derive
\begin{align*}
\left\vert\mathbb{E}\left(\int_{0}^{T-\beta } k_{3,2}(t)dt\right)\right\vert
\leq \Psi _{3}(\Vert x_{0}-y_{0}\Vert _{2}/T)
\end{align*}%
and
\begin{align*}
\left\vert\mathbb{E}\left(\int_{0}^{T-\beta } k_{3,1}(t)dt\right)\right\vert
\leq \Psi _{4}(\Vert x_{0}-y_{0}\Vert _{2}/T),
\end{align*}%
where $\Psi _{3}(\cdot )$ and $\Psi _{4}(\cdot )$ are defined as
\begin{equation}  \label{Eqn-Psi-3}
\Psi _{3}(x):=d^{4}T\lambda _{\uparrow }\lambda _{\downarrow
}^{-2}Mx+d^{4}T\lambda _{\uparrow }\lambda _{\downarrow
}^{-2}M^{2}+d^{3}T\lambda _{\uparrow }\lambda _{\downarrow }^{-1}M+\frac{1}{3%
}C_{\text{BDG}}(3)d^{\frac{7}{2}}T^{\frac{1}{2}}\lambda _{\uparrow }^{\frac{3%
}{2}}\lambda _{\downarrow }^{-2}M,
\end{equation}%
\begin{equation}  \label{Eqn-Psi-4}
\begin{split}
& \quad \Psi _{4}(x):=\left( \frac{1}{4}d^{5}T^{2}\lambda _{\uparrow
}\lambda _{\downarrow }^{-2}M+\frac{1}{2}d^{6}T^{2}\lambda _{\uparrow
}\lambda _{\downarrow }^{-3}M^{2}\right) x^{2} \\
& + \frac{1}{9}C_{\text{BDG}}(3)(d\lambda _{\uparrow }T)^{\frac{3}{2}%
}(\lambda _{\downarrow }^{-2}Md^{3}+2\lambda _{\downarrow
}^{-3}M^{3}d^{4})x+d^{5}T^{2}\lambda _{\uparrow }\lambda _{\downarrow
}^{-2}M^{2}x+d^{6}T^{2}\lambda _{\uparrow }\lambda _{\downarrow
}^{-3}M^{3}x \\
& +d^{3}T^{2}\lambda _{\uparrow }\lambda _{\downarrow
}^{-1}M x+\frac{1}{24}C_{\text{BDG}}(4)d^{5}T\lambda _{\uparrow }\lambda
_{\downarrow }^{-2}M+\frac{1}{12}C_{\text{BDG}}(4)d^{6}T\lambda _{\uparrow
}\lambda _{\downarrow }^{-3}M^{2} \\
& +\frac{2}{9}C_{\text{BDG}}(3)d^{\frac{9}{2}%
}T^{\frac{3}{2}}\lambda _{\uparrow }^{\frac{3}{2}}\lambda _{\downarrow
}^{-2}M^{2}+\frac{2}{9}C_{\text{BDG}}(3)d^{\frac{11}{2}}T^{\frac{3}{2}}\lambda
_{\uparrow }^{\frac{3}{2}}\lambda _{\downarrow }^{-3}M^{3}+\frac{1}{9}C_{%
	\text{BDG}}(3)d^{\frac{5}{2}}T^{\frac{3}{2}}\lambda _{\uparrow }^{\frac{3}{2}%
}\lambda _{\downarrow }^{-1}M \\
&+\frac{1}{2}d^{6}T^{2}\lambda _{\uparrow
}\lambda _{\downarrow }^{-3}M^{4}+\frac{1}{2}d^{4}T^{2}\lambda _{\uparrow
}\lambda _{\downarrow }^{-1}M^{2} +\frac{3}{4}d^{5}T^{2}\lambda _{\uparrow }\lambda _{\downarrow }^{-2}M^{3}.
\end{split}%
\end{equation}%

Finally, let us consider $\mathbb{E}(J^{\beta }(0,\eta (T-\beta )))$. Since
\begin{align*}
& \mathbb{E}((\eta (T-\beta )-y_{0})_{i}(\eta (T-\beta )-y_{0})_{j}) \\
=& \mathbb{E}((\eta (T-\beta )-\phi (T-\beta ))_{i}(\eta (T-\beta )-\phi
(T-\beta ))_{j} \\
& \quad +(\phi (T-\beta )-\phi (T))_{i}(\phi (T-\beta )-\phi (T))_{j}) \\
\leq & d\lambda _{\uparrow }\beta +\beta ^{2}|x_{0}-y_{0}|^{2}/T^{2},
\end{align*}%
It follows that
\begin{equation*}
\begin{split}
& \mathbb{E}(J^{\beta }(0,\eta (T-\beta ))) \\
=& \frac{d}{2}\log (2\pi \beta )+\frac{1}{2}\log \det (a(y_{0}))+\frac{1}{%
2\beta }\sum_{i,j=1}^{d}g_{ij}(y_{0})\mathbb{E}((\eta (T-\beta
)-y_{0})_{i}(\eta (T-\beta )-y_{0})_{j}) \\
\leq & \frac{d}{2}\log (2\pi \beta )+\frac{d}{2}\log \lambda _{\uparrow }+%
\frac{1}{2}d^{2}\lambda _{\uparrow }\lambda _{\downarrow }^{-1}+\frac{d}{2}%
\frac{\Vert x_{0}-y_{0}\Vert _{2}^{2}}{T}.
\end{split}%
\end{equation*}%
To conclude, let us summarize all the intermediate results, and substitute
them into \eqref{Eqn-J-beta}, we have
\begin{equation*}
J(T,x_{0})=\lim_{\beta \rightarrow 0}J^{\beta }(T-\beta ,x_{0})\leq
J_{\uparrow }(\Vert x_{0}-y_{0}\Vert _{2};T)
\end{equation*}%
where $J_{\uparrow }(\cdot ;T)$ is defined as
\begin{align}
& J_{\uparrow }(x;T):=\lambda _{\downarrow }^{-1}T\left( M+\frac{\Vert
y_{0}-x_{0}\Vert _{2}}{T}\right) ^{2}+\frac{d}{2}(\log (2\pi T))+M(d\lambda
_{\uparrow }T)^{1/2}d\lambda _{\downarrow }^{-1}  \label{Eqn-J-Uparrow} \\
& +\frac{d}{2}\lambda _{\downarrow }^{-1}Mx+\frac{1}{2}d\lambda _{\uparrow
}T\Psi _{1}(x/T)+\frac{1}{4}d\lambda _{\uparrow }T^{2}\Psi _{2}(x/T)  \notag
\\
& +\frac{1}{3}C_{\text{BDG}}(3)Md^{\frac{9}{2}}\lambda _{\downarrow
}^{-2}\lambda _{\uparrow }^{\frac{3}{2}}T^{\frac{1}{2}}+\Psi _{3}(x/T)+\Psi
_{4}(x/T)+\frac{d}{2}\log \lambda _{\uparrow }  \notag \\
& +\frac{1}{2}d^{2}\lambda _{\uparrow }\lambda _{\downarrow }^{-1}+\frac{d}{2%
}\frac{x^{2}}{T},\quad \forall \beta \in (0,T).  \notag
\end{align}%
Therefore, if we pick $D_{S}=\sup_{x\in S}\Vert x-x_{0}\Vert _{2}$, it
follows that
\begin{equation*}
p(x,T;x_{0},y)\geq \exp \left( -J_{\uparrow }(D_{S};T)\right) ,\quad \forall
x\in S,
\end{equation*}%
which ends the proof.
\end{proof}

\begin{algorithm}
	 	\caption{Computation of the lower bound $\delta_S$}
	 	\begin{algorithmic}[1]
	 		\State $D_S \gets \sup_{x\in S}\|x-x_0\|_2$.
	 		\State Evaluate $\Psi_i(D_S);i = 1,2,3,4$ by \eqref{Eqn-Psi-1},\eqref{Eqn-Psi-2},\eqref{Eqn-Psi-3} and \eqref{Eqn-Psi-4}.
	 		\State Evaluate $J_{\uparrow}(D_S;T)$ by \eqref{Eqn-J-Uparrow}.
	 		\State $\delta_{S}\gets\exp\left(-J_{\uparrow}(D_S;T)\right)$.
	 		\State \textbf{Output} $\delta_{S}$.
	 	\end{algorithmic}
 		\label{Algo-Lower-Bound}
	 \end{algorithm}

\end{document}